\documentclass[12pt,reqno
  ]{amsart}
\usepackage{amssymb,amsmath,amsthm,secdot,bbm,bm}
\usepackage{mathrsfs}
\usepackage[english]{babel}
\usepackage[LGR,T1]{fontenc} 
\usepackage{lmodern}
\usepackage[
bookmarks=true,         
bookmarksnumbered=true, 
colorlinks=true, pdfstartview=FitV, linkcolor=blue, citecolor=blue,
urlcolor=blue]{hyperref}
\usepackage[abbrev,msc-links,alphabetic]{amsrefs}
\usepackage[innermargin=1.4in,outermargin=1.4in,bottom=1.5in,marginparwidth=1.1in,marginparsep=3mm]{geometry}
\usepackage[shortlabels] {enumitem}
\usepackage{xcolor}
\usepackage[babel=true,kerning=true]{microtype}
\usepackage[capitalize,nameinlink]{cleveref}
\colorlet{darkred}{red!90!black}

\theoremstyle{plain}
\newtheorem{theorem}{Theorem}[section]

\newtheorem{corollary}[theorem]{Corollary}

\newtheorem{lemma}[theorem]{Lemma}

\newtheorem{proposition}[theorem]{Proposition}

\newtheorem{innercustomthm}{Theorem}
\newenvironment{customthm}[1]
  {\renewcommand\theinnercustomthm{#1}\innercustomthm}
  {\endinnercustomthm}

\theoremstyle{remark}

\theoremstyle{definition}
\newtheorem{remark}[theorem]{Remark}
\newtheorem*{remark*}{Remark}
\newtheorem{assumption}[theorem]{Assumption}
\newtheorem{definition}[theorem]{Definition}

\numberwithin{equation}{section}

  \newcommand{\R}{{\mathbb{R}}}
  
  \newcommand{\Rd}{{\mathbb{R}^d}}
  \newcommand{\E}{\mathbb{E}}

  \newcommand{\caa}{{\mathcal A}}
  \newcommand{\cbb}{{\mathcal B}}
  \newcommand{\sbb}{{\mathscr{B}}}

  \newcommand{\cff}{{\mathcal F}}
  \newcommand{\cgg}{{\mathcal G}}

  \newcommand{\cjj}{{\mathcal J}}

  \newcommand{\C}{{\mathcal  C}}
  \newcommand{\cpp}{{\mathcal  P}}
  
  \newcommand{\css}{{\mathcal S}}
  
  \newcommand{\cmm}{{\mathcal M}}
  \newcommand{\cnn}{{\mathcal N}}
  
  \newcommand{\cxx}{{\mathcal X}}

  \newcommand{\supp}{\mathrm{supp}\,}

  \renewcommand{\P}{{\mathbb P}}

  \newcommand{\bes}{{\mathcal B}}

  \newcommand{\wei}[1]{{\langle#1 \rangle}}

\newcommand\p{\mathfrak{p}}

\newcommand\A{\mathcal A}

\newcommand{\1}{\mathbf{1}}

\newcommand{\tand}{\quad\text{and}\quad}

\newcommand{\hs}{{H}}
\newcommand{\les}{\lesssim}

\title{Stochastic sewing in Banach space}

\author[K. L\^e]{Khoa L\^e}
\address{TU Berlin}
\email{le@math.tu-berlin.de}
\begin{document}
\begin{abstract} A stochastic sewing lemma which is applicable for  processes taking values in Banach spaces is introduced. Applications to additive functionals of fractional Brownian motion of distributional type are discussed.
\end{abstract}
\maketitle
\section{Introduction} 
\label{sec:introduction}
	The sewing lemma is an instrumental and versatile tool originated from Lyons' theory of rough paths \cite{MR1654527}.
	A specific case of the lemma can be traced back at least to the work \cite{young} of Young on the Riemann--Stieltjes integrals. Lyons utilizes Young's argument in \cite{MR1654527} to show unique extension of almost rough paths. In \cite{MR2091358}, Gubinelli  gives a general statement (Proposition 1 therein) of what we call the sewing lemma, his proof uses some elements from geometric measure theory. 
	The connection with abstract Riemann sums and the proof based on dyadic partitions, which we will adopt and develop upon herein, are due to Feyel and  de La Pradelle in \cite{MR2261056}. The broadly accepted name ``sewing lemma'' seemingly appears first in \cite{MR2261056}.

	The sewing lemma is purely analytic and provides quantifying estimates on (rough) integrals. The stochastic sewing lemma introduced by the author in \cite{MR4089788} is an extension of the sewing lemma which takes into account stochastic cancellations. 
	(See also \cite{MR4174393}*{Section 4.6} for a brief introduction on the stochastic sewing lemma.)
	Since its introduction, the stochastic sewing lemma has caught some attention and led to new interesting applications; ranging from regularization by noise problems \cites{ABLM,doi:10.1142/S0219493721400104,g2020regularisation,harang2021regularity}, numerical methods for stochastic differential equations \cites{butkovsky2021approximation,dareiotis2021quantifying,le2021taming}, rough stochastic differential equations \cite{FHL21} to averaging principle with fractional dynamics \cites{MR4124526,li2020slowfast}.
	While the sewing lemma is applicable for processes in any Banach spaces, its stochastic version from \cite{MR4089788} is only applicable for stochastic processes in $\Rd$. 
	The main purpose of this article is to extend the stochastic sewing lemma for stochastic processes taking values in any Banach spaces. Our result presented herein reconciles the two sewing lemmas while at the same time, opens up new possibilities for applications. 
	We provide one example on the spatial Besov regularity of additive functionals of fractional Brownian motion of the type $\int_0^1 f_r(B_r+\cdot)dr$ where $f$ is a time-dependent distribution on $\Rd$.  
	Closely related results along this direction include Berman's condition (\cite{MR556414}*{Theorem 28.1}) and the work \cite{doi:10.1142/S0219493721400104} in which spatial Sobolev regularity in $W^\alpha_2(\Rd)$ of local times of Gaussian processes are considered. We are able to deal with generic time dependent distributions on Besov scales with integrability indices in $(1,\infty)$.

	\smallskip
	Let us discuss on further detail.	The sewing method concerns about the convergence of the abstract Riemann sums
	\begin{align}\label{abs.rie}
		\sum_{[u,v]\in \pi}A_{u,v}
	\end{align}	
	where $A$ is a map from the simplex $\Delta:=\{(s,t)\in[0,T]^2:s\le t\}$ to a Banach space $(\cxx,|\cdot|_\cxx)$ and $\pi$ is a generic partition of $[0,T]$. We think of $A_{s,t}$ as a generalized  increment over the time interval $[s,t]$. For general Banach spaces, the sewing lemma gives  the following sufficient condition for the convergence of \eqref{abs.rie} in $\cxx$: there exist positive constants $\varepsilon,C$ and a control $w$ such that
	\begin{equation}\label{con.gub}
		|\delta A_{s,u,t}|_\cxx\le Cw(s,t)^{1+\varepsilon} \quad\forall 0\le s\le u\le t\le T,
	\end{equation}
	where $\delta A_{s,u,t}=A_{s,t}-A_{s,u}-A_{u,t}$.
	Hereafter, a control $w$ is a continuous function $w:\Delta\to[0,\infty)$ such that
	\begin{equation}
		w(s,u)+w(u,t)\le w(s,t) \quad\forall s\le u\le t.
	\end{equation}
	The sewing lemma also provides bounds on the limiting object. Let $\A_t$ denote the limit of Riemann sums \eqref{abs.rie} among partitions $\pi$ of $[0,t]$. Under \eqref{con.gub}, one has
	\begin{align}
		|\delta \A_{s,t}-A_{s,t}|_\cxx\le C w(s,t)^{1+\varepsilon} \quad\forall (s,t)\in \Delta
	\end{align}
	for some constant $C$, where $\delta \A_{s,t}=\A_t-\A_s$. 
	
	When $\cxx=[L^m(\Omega,\cff,(\cff_t),\P)]^d$ for some filtered probability space $(\Omega,\cff,(\cff_t),\P)$ and integers $d\ge1$, $m\ge2$, the stochastic sewing lemma from \cite{MR4089788} gives the following sufficient condition for the convergence of \eqref{abs.rie} in $[L^m(\Omega)]^d$: $A_{s,t}$ is $\cff_t$-measurable and there are positive constants $\varepsilon,C$ such that
	\begin{align}\label{con.kle}
		\|\delta A_{s,u,t}\|_{L^m(\Omega)}\le C|t-s|^{\frac12+\varepsilon}\tand\|\E(\delta A_{s,u,t}|\cff_s)\|_{L^m(\Omega)}\le C|t-s|^{1+\varepsilon}
	\end{align}
	for every $0\le s\le u\le t\le T$. The corresponding estimates on $\A$ supplied by the stochastic sewing lemma are
	\begin{align}
		\|\delta\A_{s,t}-A_{s,t}\|_m\le C|t-s|^{\frac12+\varepsilon}
		\tand
		\|\E\left(\delta\A_{s,t}-A_{s,t}|\cff_s\right)\|_m\le C|t-s|^{1+\varepsilon}.
	\end{align}

	When $w(s,t)=t-s$ and $\cxx=[L^m(\Omega,\cff,(\cff_t),\P)]^d$, one can compare the two results: \eqref{con.kle} requires less regularity on $\delta A$ but instead impose adaptiveness and an additional regularity condition on the conditional quantity $\E(\delta A_{s,u,t}|\cff_s)$. When $A$ is deterministic, \eqref{con.kle} evidently deduces to \eqref{con.gub}. 
	On the other hand, the stochastic sewing lemma  is restricted to the space $[L^m(\Omega,\cff,(\cff_t),\P)]^d$ whereas the sewing lemma applies for general Banach spaces. 
	Additionally, a curious feature of the stochastic sewing lemma is the exponent $1/2$ in \eqref{con.kle}. 
	These differences are due to the upper bound in the Burkholder--Davis--Gundy (BDG) inequality in $\Rd$ used in \cite{MR4089788}.

	\smallskip
	The current article gives an extension of the stochastic sewing lemma when $\cxx$ takes the form $L^m((\Omega,\cff,(\cff_t),\P);V)$ for some separable Banach space $(V,|\cdot|_V)$. 
	The space $L^m((\Omega,\cff,(\cff_t),\P);V)$ contains all $V$-valued random variables which have finite $m$-th moment.
	To state the result, however, we need to introduce another feature on the geometric properties of $V$.

	As explained previously, the exponent $1/2$  in \eqref{con.kle} is tied to the upper bound in the BDG inequalities in $\Rd$. This inequality estimates moments of a martingale by the moments of its quadratic variation. While BDG inequality is easily extended to Hilbert spaces, such inequality is not available for general Banach spaces.
	Nevertheless, it turns out that many Banach spaces possess a certain variant BDG upper bound. Such property is described by the concept of \textit{martingale type}, a  feature originated from the study of geometry of Banach spaces.

	Let $(f_h)_{h=0}^N$ be a sequence of integrable $V$-valued random variables and $\{\cff_h\}_{h=0}^N$ be a non-decreasing sequence of subfields of $\cff$. We say that $(f_h,\cff_h)_{h=0}^N$, or simply $(f_h)_{h=0}^N$, is a martingale if
	\[
		\E(df_{h+1}|\cff_h)=0 \quad\forall h\ge0, \quad\text{where} \quad df_h=f_h-f_{h-1}.
	\]
	When $f_h$ is $L_m$-integrable for every $h$, we say that $(f_h)_{h=0}^N$ is an $L_m$-integrable martingale.

	\begin{definition}\label{def.marttype}
		We say that $V$ has \emph{martingale type} $\p\in[1,2]$ if for some $m\in(1,\infty)$, there exists a constant $C_{m,\p,V}$ such that
		\begin{equation}\label{ineq.type}
			\|f_N\|_{L_m(V)}\le C_{m,\p,V}\left\|\left(|f_0|_V^\p+\sum_{n=1}^N|df_n|_V^\p\right)^{1/\p}\right\|_{L_m}
		\end{equation}
		for every $L_m$-integrable $V$-valued martingale $(f_n)_{n=0}^N$.
	\end{definition}
	We note that every Banach space has trivial martingale type $\p=1$, in which \eqref{ineq.type} follows from the triangle inequality. Hence, if $\p\neq1$, we say that $V$ has non-trivial martingale type. As an example, the Besov space $\bes^\alpha_{p,q}(\Rd)$, $\alpha\in\R$, $(p,q)\in[1,\infty)$ has martingale type $\p=\min(2,p,q)$, see \cref{prop.ex.type} for more examples.
	\begin{customthm}{A}\label{thm.meta}
		Let $w$ be a control, $m\ge2$ and $(V,|\cdot|_V)$ be a Banach space with martingale type $\p$. Let $A:\Omega\times \Delta\to V$ be a measurable map such that $A_{s,t}$ is $\cff_t$-measurable for every $(s,t)\in \Delta$.
		Suppose that there are positive constants $\varepsilon,C$ such that
		\begin{align}\label{con.meta}
			\||\E(\delta A_{s,u,t}|\cff_s)|_V\|_{L^m(\Omega)}\le Cw(s,t)^{1+\varepsilon}
			\tand
			\||\delta A_{s,u,t}|_V\|_{L^m(\Omega)}\le Cw(s,t)^{\frac1\p+\varepsilon}
		\end{align}
		for every $0\le s\le t\le T$.
		Then the Riemann sums \eqref{abs.rie} converge in $L^m(\Omega;V)$.		
	\end{customthm}
	\cref{thm.meta} reconciles the sewing lemma and its stochastic version from \cite{MR4089788}.
	Indeed, it is evident that when $A$ is deterministic, condition \eqref{con.meta} deduces to \eqref{con.gub} and hence, \cref{thm.meta} deduces to the sewing lemma in its full generality. 
	On the other hand, the space $\Rd$ has martingale type $\p=2$ and hence \cref{thm.meta} deduces to the stochastic sewing lemma from  \cite{MR4089788} when $V=\Rd$.
	The exponent $1/2$ in \eqref{con.kle} is revealed in \eqref{con.meta} as $1/\p$, where $\p$ is the martingale type of $V$.  
	\cref{thm.meta} follows from a more general result, \cref{prop.VSSL} herein.

	\smallskip
	As an application, we apply the stochastic sewing lemma in Banach spaces to study a class of additive functionals of fractional Brownian motion. To be more precise, let $B$ be a fractional Brownian motion in $\Rd$ and $f$ be a time-dependent distribution in the space $L^\theta([0,T];\bes^\alpha_{p,1}(\Rd))$, $\alpha\in\R$, $\theta,p\in(1,\infty)$. Consider additive functionals of the type
	\begin{equation}
		(t,x)\mapsto I[f]_t(x):=\int_0^t f_r(B_r+x)dr.
	\end{equation}
	Such additive functional has intimate connections with the local times of fractional Brownian motion and its regularity properties are directly related to regularizing effect of fraction Brownian motion (\cite{MR3505229,MR4089788}).
	Two problems to be discussed are defining $I[f]$ and establishing its spatial regularity.
	Because $f_r$ is only a distribution for a.e. $r$, the composition $f_r(B_r+x)$ is not a priori well-defined, so is the functional $f\mapsto I[f]$.
	To overcome this issue, we define $f\mapsto I[f]$ as the continuous extension on space of smooth functions with respect to the topology generated by $L^\theta([0,T];\bes^\alpha_{p,1}(\Rd))$. 
	This means that whenever defined, for every sequence of smooth functions $(f^n)_n$ which converges to $f$ in $L^\theta([0,T];\bes^\alpha_{p,1}(\Rd))$, we have
	\begin{equation}\label{intro.limfn}
		\lim_n\int_0^t f^n_r(B_r+x)dr=I[f]_t(x)
	\end{equation}
	in an appropriate sense. 
	It turns out that the limit  in \eqref{intro.limfn} holds even when the spatial regularity is larger than $\alpha$. Hence, \eqref{intro.limfn} also quantifies the regularizing effect of $B$ and provides spatial regularity for $I[f]$, see \cref{thm.fbm.besov} herein.
	Although the method can be extended to other stochastic processes, as one has seen from the aforementioned applications of the stochastic sewing lemma, we restrict herein to fractional Brownian motion. An advantage is  that it keeps our presentation concise while at the same time exhibit the method in different scales by varying the Hurst parameter.

	\smallskip
	The structure of the paper is as follows. In \cref{sec:vector_valued_stochastic_processes}, we briefly summarize necessary facts on vector valued stochastic processes as well as the concept of martingale type. The stochastic sewing lemma in Banach spaces is stated and proved in \cref{sec:stochastic_sewing_lemma}. \cref{sec:additive_functionals_of_fractional_brownian_motion} discusses an application to additive functional of fractional Brownian motion. 
	\cref{sec:a_continuity_criterion_with_controls} presents an extension of Kolmogorov continuity criterion with controls which can be used to construct a continuous modification of the process constructed from \cref{thm.meta}. \cref{sec:besov_spaces} contains some auxiliary estimates.

	\smallskip

	We close the introduction with a list of notation.	
		$\css(\Rd)$ is the Schwartz space,	
		$\bes^\alpha_{p,q}(\Rd)$ is the Besov space, $W^\alpha_p(\Rd)$ is the fractional Sobolev space, $\C^\alpha(\Rd):=\bes^\alpha_{\infty,\infty}(\Rd)$ is the H\"older--Zygmund space (see \cite{bahouri} for precise definitions). $L^p=L^p(\Rd)$ is the Lebesgue space while $L_m=L^m(\Omega)$ is the moment space. Deterministic norms are denoted by $|\cdot|$. For example, the norms in $\bes^\alpha_{p,q}(\Rd)$ and $L^p(\Rd)$ are denoted by $|\cdot|_{\bes^\alpha_{p,q}}$ and $|\cdot|_{L^p}$ respectively. The norm in the moment space $L_m(\Omega)$ is denoted by $\|\cdot\|_m$. 
		For an $V$-valued random variable $X$, where $V$ is a Banach space, the moment norm of $X$ is denoted by $\|X\|_{V;m}:=\left\||X|_V\right\|_m$. 
		For each $q\in[1,\infty]$, $q'$ denotes the H\"older conjugate of $q$, defined by $\frac1q+\frac1{q'}=1$. We always use the convention that $1/\infty=0$. 
		The notation $\lesssim$ means $\le C$ for some universal constant $C>0$ whose value can change from one line to another.

	\section{Vector valued stochastic processes} 
	\label{sec:vector_valued_stochastic_processes}
		Throughout the article, $T>0$ is a fixed finite time and $(\Omega,\cff,\{\cff_t\}_{t\in[0,T]},\P)$ is a complete filtered probability space such that $\cff_0$ contains the $\P$-null sets. The expectation with respect to $\P$ is denoted by $\E$ while the conditional expectation with respect to $\cff_s$ is denoted by $\E_s$. 
		For a topological space $S$, the Borel $\sigma$-field on $S$ is denoted by  $\sbb(S)$. 

		Let $(V,|\cdot|_V)$ be a separable Banach space and $\sbb(V)$ be its Borel $\sigma$-field. 
		A \emph{$V$-valued random variable} $Y$ is a measurable map $Y:(\Omega,\cff)\to(V,\sbb(V))$. Since $V$ is separable, weakly measurable maps are strongly measurable (this is Pettis measurability theorem \cite{MR3617205}*{Theorem 1.1.20}) and we will henceforth always write ``measurable'' for ``strongly measurable''.

		For $m\in[1,\infty)$, we say that a $V$-valued random variable $Y$ is $L_m$-integrable if $\E(|Y|_V^m)$ is finite. When $m=1$, we simply say $Y$ is integrable. In this case, $\E Y$ is a well-defined element in $V$ and one has
		\[
			|\E Y|_V\le \E(|Y|_V).
		\] 
		The space of all $L_m$-integrable $V$-valued random variables is denoted by $L_m(V)$. As commonly practiced, the dependence on $\Omega$ is omitted. Nevertheless, when it is necessary to emphasize the stochastic basis, we write $L^m((\Omega,\cff,\{\cff_t\},\P);V)$ or $L^m(\Omega;V)$  for $L_m(V)$. The norm on $L_m(V)$ is defined by
		\[
			Y\mapsto \|Y\|_{V;m}:=\left(\E|Y|^m_V\right)^{1/m}.
		\]
		Let $\cgg$ be a sub $\sigma$-field of $\cff$ and $Y$ be a $V$-valued random variable. The conditional expectation of $Y$ with respect to $\cgg$, whenever exists, is a (unique) $V$-valued random variable $Z=\E(Y|\cgg)$ such that
		\[
			\E(Y\1_J)=\E(Z\1_J)\quad\forall J\in\cgg.
		\]
		If $Y$ is integrable, we have from \cite{MR3617205}*{Lemma 2.6.19} that $|\E(Y|\cgg)|_V\le\E(|Y|_V|\cgg)$.
		We also denote
		\[
			\|Y|\cgg\|_{V;m}:=\left[\E\left(|Y|^m_V|\cgg\right)\right]^{1/m}
		\]
		which is a $\cgg$-random variable.
		When $V=\R^k$ for some integer $k\ge1$, we simply write $\|\cdot\|_m$ and $\|\cdot|\cgg\|_m$ respectively for $\|\cdot\|_{\R^k,m}$ and $\|\cdot|\cgg\|_{\R^k,m}$. 
		For further information on integration and conditional expectation of Banach-valued random variables, we refer to \cite{MR3617205}*{Sections 1.2 and 2.6}.

		We recall \cref{def.marttype} of martingale type. Furthermore,  if \eqref{ineq.type} holds for some $m\in(1,\infty)$ then it holds for \emph{all} $m\in(1,\infty)$ (see \cite{MR3617205}*{Proposition 3.5.27}). Therefore, we can say that $V$ has martingale type $\p$ if \eqref{ineq.type} holds for every $m\in(1,\infty)$ and every $L_m$-integrable martingale $(f_n)_{n=0}^N$. 

				Although not being used, it worths noting the following relation between martingale type and smoothness of Banach spaces, due to Pisier.
				\begin{proposition}[{\cite[Corollary 4.22]{MR3617459}}] Let $V$ be a Banach space and $\p$ be in $[1,2]$. $V$ has martingale type $\p$ if and only if $V$ is $(\p,C)$-smooth for some constant $C$, i.e. there is an equivalent norm $|\cdot|$ on $V$ such that 
		\begin{align*}
			2^{-1}(|x+y|^\p+|x-y|^\p)\le |x|^\p+C^\p|y|^\p, \quad \forall x,y\in V.
		\end{align*}				 	
				\end{proposition}
		
		In our considerations, it is useful to observe that \eqref{ineq.type} also holds with conditional expectations. To be more precise, suppose that $(f_h)_{h=0}^N$ is an $L_m$-integrable $V$-valued martingale with respect to a filtration $\{\cff_h\}_{h=0}^N$ and $\cgg$ is a sub $\sigma$-field of $\cff_0$. Then
		\begin{equation}\label{ineq.cd.type}
			\|f_N|\cgg\|_{V;m}\le C_{m,\p,V}\left\|\left(|f_0|_V^\p+\sum_{n=1}^N|df_n|_V^\p\right)^{1/\p}\Bigg|\cgg\right\|_{L_m} \quad\textrm{a.s.}
		\end{equation}
		where $C_{m,\p,V}$ is the same constant in \eqref{ineq.type}. Indeed, let $G$ be a measurable set in $\cgg$. It is easily checked that $(f_h\1_G)_{h=0}^{N}$ is an $L_m$-integrable $V$-valued martingale. Hence, applying \eqref{ineq.type} gives
		\[
			\||f_N|_V\1_G\|_{L_m}\le C_{m,\p,V}\left\|\left(|f_0|_V^\p+\sum_{n=1}^N|df_n|_V^\p\right)^{1/\p}\1_G\right\|_{L_m}.
		\]
		Since this inequality holds for every $G$ in $\cgg$, one can apply \cite{MR3617205}*{Lemma 2.6.15} to obtain \eqref{ineq.cd.type}.

		Hereafter, the following assumption will be enforced.		
		\begin{assumption}\label{con.Vtype}
			$V$ has non-trivial martingale type $\p\in(1,2]$. 
		\end{assumption}
		As explained previously, \cref{con.Vtype} ensures that inequalities \eqref{ineq.type} and \eqref{ineq.cd.type} hold for \textit{every} $m\in(1,\infty)$, every $L_m$-integrable $V$-valued martingale $(f_h,\cff_h)_{h=0}^N$ and every subfield $\cgg$ of $\cff_0$.

		We collect a few known examples Banach spaces with non-trivial type $\p$.
		\begin{proposition}[Examples of Banach spaces with non-trivial type]\label{prop.ex.type} Let $p,q\in(1,\infty)$ and $\alpha\in\R$.
			\begin{enumerate}[(a)]
				\item\label{t1}  Every Hilbert space has martingale type $2$.
				\item\label{t2} Lebesgue space $L^p(\Rd)$ has martingale type $\p=\min\{2,p\}$.
				\item\label{t3} Besov space $\bes^\alpha_{p,q}(\Rd)$ has martingale type $\p=\min\{2,p,q\}$.
				\item\label{t4} Triebel--Lizorkin space $F^\alpha_{p,q}(\Rd)$ has martingale type $\p=\min\{2,p,q\}$.
				\item\label{t5} Sobolev space $W^\alpha_p(\Rd)$ has martingale type $\p=\min\{2,p\}$. 
			\end{enumerate}
		\end{proposition}
		\begin{proof}
			We mostly point to existing literature and refer to \cites{MR3617205,MR3752640} for relevant definitions of (Rademacher) type and cotype, martingale cotype and UMD property.
			\ref{t1} comes from the Pythagorian identity 
			\begin{equation*}
				\|f_N\|_{L_2(V)}^2=\|f_0\|_{L_2(V)}^2+\sum_{n=1}^N\|df_n\|_{L_2(V)}^2
			\end{equation*}
			which is valid for Hilbert space $V$ and any martingale $(f_n)_{n=0}^N$.
			\ref{t2} is proved in \cite{MR3617205}*{Proposition 3.5.30}. 
			The Besov space $\bes^\alpha_{p,q}(\Rd)$ and $F^\alpha_{p,q}(\Rd)$ have UMD-property (\cite{MR942266}*{Theorem 6.3}) and have (Rademacher) type $\min\{2,p,q\}$. 
			On UMD spaces, type and martingale type are equivalent, \cite{MR3617205}*{Proposition 4.3.13}. 
			This implies \ref{t3} and \ref{t4}.  When $\alpha$ is not an integer, \ref{t5} follows from \ref{t3} because $W^\alpha_p(\Rd)=\bes^\alpha_{p,p}(\Rd)$. When $\alpha$ is a non-negative integer, $W^\alpha_p(\Rd)$ has  type $\min\{2,p\}$ and  cotype $\max\{2,p\}$ (\cite{MR835828}). Additionally, $W^\alpha_p(\Rd)$ is UMD (\cite{MR3617205}*{Example 4.2.18}), hence it has martingale type $\min\{2,p\}$. When $\alpha$ is a negative integer, by duality (\cite{MR3617205}*{Prop. 3.5.29}), the martingale type of $W^\alpha_p(\Rd)$ is the same as the martingale cotype of $W^{-\alpha}_{p'}(\Rd)$, $\frac1p+\frac1{p'}=1$. 
			We have deduced \ref{t5} for all possible values of $\alpha$.
		\end{proof}		
		Further examples of Banach spaces of non-trivial martingale type are the Besov spaces of modeled distributions, see \cite{MR4261328}.		
		We conclude the section by  some useful estimates for adapted sequences of $V$-valued random variables.
		\begin{lemma}\label{lem.cBDG}
			Let $m$ be a real number and $n$ be an extended real number such that $n,m\ge\p$. 
			Let $\{y_k\}_{k\ge0}$ be a sequence of random variables in $V$ which is $L_m$-integrable and adapted to a filtration $\{\cff_k\}_{k\ge0}$. For each $k\ge1$, we denote $\E_{k-1}y_k=\E(y_k|\cff_{k-1})$. 
			Let $\cgg$ be a $\sigma$-field such that $\cgg\subset\cff_0$.
			Then for every $N\ge0$
			\begin{align}\label{ineq.DBDB.center}
				\|\|\sum_{k=0}^Ny_k|\cgg\|_{V;m}\|_n
				&\le\sum_{k=1}^N\|\|\E_{k-1}y_k|\cgg\|_{V;m}\|_n 
				\\&\quad+C_{m,\p,V}\left(\|\|y_0|\cgg\|_{V;m}\|_n^\p+ \sum_{k=1}^N\|\|y_k-\E_{k-1}y_k|\cgg\|_{V;m}\|_n ^\p\right)^{\frac1\p}
				\nonumber
			\end{align}
			and
			\begin{equation}\label{ineq.DBDB}
				\|\|\sum_{k=0}^Ny_k|\cgg\|_{V;m}\|_n\le\sum_{k=1}^N\|\|\E_{k-1}y_k|\cgg\|_{V;m}\|_n +2C_{m,\p,V}\left(\sum_{k=0}^N\|\|y_k|\cgg\|_{V;m}\|_n ^\p\right)^{\frac1\p}.
			\end{equation}
		\end{lemma}
		\begin{proof}
			Define $f_{-1}=0$, $f_0=y_0$ and $f_k=y_0+\sum_{h=1}^k(y_h-\E_{h-1}y_h)$ for each $k\ge1$. Then $\{f_k,\cff_k\}_{k\ge0}$ is a martingale with values in $V$ which is $L_m$-integrable. From the Doob's decomposition
			\[
				\sum_{k=0}^Ny_k=\sum_{k=1}^N\E_{k-1}y_k+f_N
			\]
			we obtain by triangle inequality that
			\[
				\|\|\sum_{k=0}^Ny_k|\cgg\|_{V;m}\|_n \le \sum_{k=1}^N\|\|\E_{k-1}y_k|\cgg\|_{V;m}\|_n +\|\|f_N|\cgg\|_{V;m}\|_n.
			\]
			Since $m,n\ge\p$, from \eqref{ineq.cd.type}, we apply Minkowski inequality to obtain that
			\begin{align*}
				\|\|f_N|\cgg\|_{V;m}\|_n\le C_{m,\p,V}\left\|\left\|\left(\sum_{k=0}^N|df_k|_V^\p\right)^{1/\p}\Bigg|\cgg\right\|_{m}\right\|_n
				\le C_{m,\p,V}\left(\sum_{k=0}^N\|\|df_k|\cgg\|_{V;m}\|_n ^\p\right)^{1/\p}.
			\end{align*}
			This leads to the following inequality
			\begin{multline*}
				\|\|\sum_{k=0}^Ny_k|\cgg\|_{V;m}\|_n \le \sum_{k=1}^N\|\|\E_{k-1}y_k|\cgg\|_{V;m}\|_n +C_{m,\p,V}\left(\sum_{k=0}^N\|\|df_k|\cgg\|_{V;m}\|_n ^\p\right)^{1/\p}.
			\end{multline*}
			Since $df_0=y_0$ and $df_k=y_k-\E_{k-1}y_k$ for $k\ge1$, the previous inequality is equivalent to \eqref{ineq.DBDB.center}. 
			The estimate \eqref{ineq.DBDB} follows from \eqref{ineq.DBDB.center} if we can show that 
			\begin{equation*}
				\|\|y_k-\E_{k-1}y_k|\cgg\|_{V;m}\|_n\le 2\|\|y_k|\cgg\|_{V;m}\|_n
			\end{equation*}
			for every $k\ge1$.
			This is a trivial consequence of triangle inequality and the estimate $|\E_{h-1}y_h|_V\le\E_{h-1}(|y_h|_V)$.
		\end{proof}
	
\section{Stochastic sewing lemma in Banach spaces} 
\label{sec:stochastic_sewing_lemma}

	\subsection{The results} 
	\label{sub:results}
		We recall that $(V,|\cdot|_V)$ is a Banach space with martingale type $\p\in(1,2]$, $\Delta$ is the simplex $\{(s,t)\in[0,T]^2:s\le t\}$ for some fixed $T>0$ and $w$ is a continuous control on $\Delta$.
		We denote by $\Delta_2$ the set $\{(s,u,t)\in[0,T]^3:s\le u\le t\}$ and by $\cpp(I)$ the set of partitions of some interval $I$. For each $\pi\in\cpp(I)$, we define its mesh size with respect to the control $w$ as  $|\pi|_w:=\sup_{[u,v]\in\pi}w(u,v)$.
		
		\begin{theorem}[Stochastic Sewing Lemma in Banach spaces]
		\label{prop.VSSL}
			Let $\p\le m\le n\le \infty$ be fixed, $m<\infty$. 
			Let $A:\Omega\times \Delta\to V$ be $\cff\otimes\sbb(\Delta)/\sbb(V)$-measurable such that $A_{s,s}=0$ and $A_{s,t}$ is $\cff_t$-measurable for every $(s,t)\in \Delta$.
			For each $t\in[0,T]$, define
			\begin{equation}\label{def.caa}
				\A_t=\lim_{\pi\in\cpp([0,t]),|\pi|_w\downarrow0}\sum_{[u,v]\in \pi}A_{u,v}
			\end{equation}
			whenever the limit exists in probability.  
			Suppose that there are constants $\Gamma _1, \Gamma_2\ge0$ such that for any $(s,u,t)\in\Delta_2$ 
			\begin{equation}\label{con.dA1}
			\|\E_s[\delta A_{s,u,t}]\|_{V;n}\leq \Gamma_1w(s,t)^{1+\varepsilon_1}
			\end{equation} 
			and
			\begin{equation}\label{con.dA2}
				\|\|\delta A_{s,u,t}-\E_s \delta A_{s,u,t}|\cff_s\|_{V;m}\|_n\leq \Gamma_2w(s,t)^{\frac1\p+\varepsilon_2}.
			\end{equation}
			Then, we have
			\begin{enumerate}[(a)]
				\item\label{cl:a} $\caa_t$ is well-defined and $\caa_t-A_{0,t}$ is $L_m$-integrable for each $t\in[0,T]$, $\caa$ is $\{\cff_t\}$-adapted,
				\item\label{cl:b} there exists constant $C>0$ such that for every $(s,t)\in \Delta$,
				\begin{equation}\label{est.A1}
					\|\E_s(\A_t-\A_s-A_{s,t})\|_{V;n}\leq C \Gamma_1w(s,t)^{1+\varepsilon_1}
				\end{equation}
				and
				 \begin{align}\label{est.A2}
				 	\|\|\A_t-\A_s-A_{s,t}|\cff_s\|_{V;m}\|_n\leq C \Gamma_1w(s,t)^{1+\varepsilon_1} 
				 	+C \Gamma_2 w(s,t)^{\frac1\p+\varepsilon_2}.
				 \end{align}
			\end{enumerate}
		\end{theorem}
		
		From \eqref{est.A1} and \eqref{est.A2}, one can derive a rate of convergence of the Riemann sums.
		\begin{corollary}[Riemann sum approximation]\label{prop.rate}
			Assume that the hypotheses of Theorem \ref{prop.VSSL} are satisfied. Let $(s,t)\in \Delta$ be fixed and $\pi$ be a partition of $[s,t]$. Define the Riemann sum
			\[
				A^\pi=\sum_{[u,v]\in\pi}A_{u,v}.
			\]
			Then, there exists a constant $C>0$ such that 
			\begin{equation}\label{est.rate}
				\|\delta\caa_{s,t}-A^\pi\|_{V;m}\le C \Gamma_1|\pi|^{\varepsilon_1}_w w(s,t)+C \Gamma_2 |\pi|^{\varepsilon_2}_ww(s,t) ^{\frac1\p}.
			\end{equation}
		\end{corollary}
		
		The following result can be considered as an extension of the Doob--Meyer decomposition for vector valued stochastic processes.
		
			\begin{theorem}\label{thm.DM}
				Suppose that the hypotheses of \cref{prop.VSSL} hold. In addition, assume that $A$ is integrable and there are constants $\Gamma_3\ge0$, $\varepsilon_3>0$ such that
				\begin{equation}
					\label{con:dA3}\|\|(\E_s-\E_u)A_{u,t}|\cff_s \|_{V;m}\|_n\le \Gamma_3 w(s,t)^{\frac1\p+\varepsilon_3}\,,
				\end{equation}
				for every $(s,u,t)\in \Delta_2$.
				Then, there exist stochastic processes $\cmm,\cjj:\Omega\times[0,T]\to V$ and positive constants $C_1,C_2,C_3$  satisfying the following properties
				\begin{enumerate}[(\roman*)]
					\item\label{cl:amj} $\cmm,\cjj$ are $\{\cff_t\}$-adapted, $L_m$-integrable and $\caa_t=\cmm_t+\cjj_t$ a.s. for every $t\in[0,T]$,
				 	\item\label{cl:m} $(\cmm_s)_{0\le s\le T}$ is an $\{\cff_t\}$-martingale with $\cmm_0=0$,
				 	\item \label{cl:est.m} for any $(s,t)\in \Delta$,
				 	\begin{equation}
		 				\label{est:M} \|\|\cmm_t-\cmm_s-A_{s,t}+\E_sA_{s,t}|\cff_s\|_{V;m}\|_n\le C_2 \Gamma_2w(s,t) ^{\frac1\p+\varepsilon_2}+C_3 \Gamma_3w(s,t)^{\frac1\p+\varepsilon_3}\,,
		 			\end{equation}
		 			\item\label{cl:est.j} for any $(s,t)\in \Delta$,
				 	 	\begin{align}
				 	 		\label{est:J}\|\| \cjj_t-\cjj_s-\E_s A_{s,t}|\cff_s\|_{V;m}\|_n \le C_1 \Gamma_1w(s,t)^{1+\varepsilon_1}+C_3 \Gamma_3w(s,t)^{\frac1\p+\varepsilon_3}\,,
				 	 	\end{align}
				 	\item \label{cl:est.j'} for any $(s,t)\in \Delta$,
				 	\begin{equation}
		 				\label{est:J'} \| \E_s(\cjj_t-\cjj_s-A_{s,t})\|_{V;n}\le C_1 \Gamma_1w(s,t)^{1+\varepsilon_1}\,.
		 			\end{equation}
				 \end{enumerate} 
				Given $A$, we have the following characterizations.
				\begin{enumerate}[resume*]
				 	\item\label{ch1} The pair of processes $(\cmm,\cjj)$ is characterized uniquely by the set of properties \ref{cl:amj}, \ref{cl:m}, \ref{cl:est.m} or, alternatively by the set of properties \ref{cl:amj}, \ref{cl:m}, \ref{cl:est.j}. 
				 	\item\label{ch2} The process $\cmm$ is characterized uniquely by \ref{cl:m} and \ref{cl:est.m}. 
				 	\item\label{ch3} The process $\cjj$ is characterized uniquely by \ref{cl:est.j}, \ref{cl:est.j'} and the fact that $\cjj$ is $\{\cff_t\}$-adapted and $\cjj_0=0$.
				\end{enumerate} 
				Furthermore, for every fixed $t\in[0,T]$ and any partition $\pi=\{0=t_0<t_1<\cdots<t_N=t\}$ of $[0,t]$, define the Riemann sums
				\begin{equation*}
						M^\pi_{t}:=\sum_{i=0}^{N-1}\left(A_{t_i,t_{i+1}}-\E_{t_i} A_{t_i,t_{i+1}}\right)
						\quad\textrm{and}\quad
						J^\pi_{t}:=\sum_{i=0}^{N-1}\E_{t_i}A_{t_i,t_{i+1}}\,.
				\end{equation*}
				Then $\{M^\pi_{t}-A_{0,t}+\E_0A_{0,t}\}_\pi$ and $\{J^\pi_{t}-\E_0A_{0,t}\}_\pi $  
				converge to $\cmm_t-A_{0,t}+\E_0A_{0,t}$ and $\cjj_t-\E_0A_{0,t}$ respectively in $L_m$ as $|\pi|_w$ goes to 0.
			\end{theorem}
			\begin{remark}
				Relations between \cref{prop.VSSL,thm.DM} and It\^o calculus are described in \cite{MR4089788}.
				The conditional norms (with $n=\infty$) in \cref{prop.VSSL} play an indispensable role in \cite{FHL21}. An immediate application of \cref{thm.DM} which is described in \cite{FHL21} is the decomposition of a stochastic controlled rough path as the sum of a martingale and a controlled rough path. We refer to the afore-mentioned references for further detail.
			\end{remark}

	\subsection{Proofs} 
	\label{sub:proofs}
		\begin{lemma}\label{lem.uniq}
			Let $w,\eta:\Delta\to\R_+$ be functions such that $w$ is super-additive and  $\lim_{|t-s|\to0}\eta(s,t)=0$.
			Let $(R_{s,t})_{(s,t)\in\Delta }$ be a two-parameter stochastic process with values in $V$ such that $R$ is $\{\cff_t\}$-adapted and satisfies
				\begin{equation}\label{con.R0}
					\|R_{s,t}\|_{V;\p}\le \left(w(s,t)\eta(s,t)\right)^{1/\p}
					\quad\textrm{and}\quad
					\|\E_s R_{s,t}\|_{V;\p}\le w(s,t)\eta(s,t).
				\end{equation}
			Then
			\begin{equation}\label{lim.Rpi}
				\lim_{|\pi|\to0}\sum_{[u,v]\in\pi}R_{u,v}=0 \quad\textrm{in}\quad L_\p(V),
			\end{equation}
			where $|\pi|=\sup_{[u,v]\in \pi}|v-u|$.
			In particular, if additionally $R$ is additive, i.e. $R_{s,u}+R_{u,t}=R_{s,t}$ a.s. for all $(s,u,t)\in \Delta_2$, then $R$ is identically $0$.
		\end{lemma}
		\begin{proof}
			Let $\varepsilon\in(0,1)$ be fixed. Let $\pi$ be a partition of $[0,T]$ such that $\sup_{[s,t]\in \pi}\eta(s,t)\le \varepsilon$. 
			Applying estimate \eqref{ineq.DBDB} (with $m=n=\p$), we see that
			\[
				\|\sum_{[u,v]\in\pi}R_{u,v}\|_{V;\p}\lesssim \left(\sum_{[u,v]\in\pi}\|R_{u,v}\|_{V;\p}^\p\right)^{1/\p}+\sum_{[u,v]\in\pi}\|\E_uR_{u,v}\|_{V;\p}.
			\]
			Using super-additivity of $w$ and condition \eqref{con.R0}, we have
			\begin{align*}
				\sum_{[u,v]\in\pi}\|R_{u,v}\|_{V;\p}^\p\le \varepsilon \sum_{[u,v]\in \pi}w(u,v)\le \varepsilon w(0,T).
			\end{align*}
			In a similar way, we have $\sum_{[u,v]\in\pi}\|\E_uR_{u,v}\|_{V;\p}\le \varepsilon w(0,T)$. It follows that $$\|\sum_{[u,v]\in\pi}R_{u,v}\|_{V;\p}\lesssim \varepsilon^{1/\p}.$$ This implies \eqref{lim.Rpi}. 

			Suppose additionally that $R$ is additive. For each $(s,t)\in \Delta$ and any partition $\pi$ of $[s,t]$, we have by additivity $\sum_{[u,v]\in \pi}R_{u,v}=R_{s,t}$ a.s.
			Hence, \eqref{lim.Rpi} implies that $R_{s,t}=0$. 
		\end{proof}
		
		Let $w$ be a control.
		For each $(s,t)\in \Delta$, define
		\[
			u=\inf\{r\in[s,t]:w(s,r)\ge\frac12 w(s,t)\}
		\]
		and call $u$ as the $w$-midpoint of $[s,t]$. Since $t$ trivially belongs to the set defining $u$ above, such a point always exists and is uniquely defined.
		If $u$ is a $w$-midpoint of $[s,t]$, then it follows from the continuity of $w$ that 
		\begin{equation}\label{est.wsut}
		 	w(s,u)\le\frac12w(s,t) \quad\textrm{and}\quad w(u,t)\le\frac12w(s,t).
		 \end{equation} 
		Indeed, by definition, there exists a sequence $\{r_j\}_j\subset[s,t]$ decreasing to $u$ such that $w(s,r_j)\ge\frac12w(s,t)$. Since $w$ is super-additive, $w(s,r_j)+w(r_j,t)\le w(s,t)$. This implies that $w(r_j,t)\le \frac12w(s,t)$. Since $w$ is continuous, $\lim_jw(r_j,t)=w(u,t)$. This implies that $w(u,t)\le\frac12 w(s,t)$.
		We show $w(s,u)\le\frac12 w(s,t)$ by contradiction. Suppose that $w(s,u)>\frac12 w(s,t)$. By continuity, there is $\varepsilon>0$ such that $w(s,r)>\frac12 w(s,t)$ for every $r\in(u- \varepsilon,u)$. This contradicts with the definition of $u$.

		Let $(s,t)$ be in $\Delta$. For each integer $h\ge0$, we define $h$-dyadic points of $[s,t]$ with respect to the control $w$ in the following way. Define $d^0_0(s,t)=s$ and $d^0_1(s,t)=t$. For each integer $i=0,\dots,2^{h+1}$, we set $d^{h+1}_i(s,t)=d^h_{i/2}(s,t)$ if $i$ is even and $d^{h+1}_i(s,t)$ equal to the $w$-midpoint of $[d^h_{(i-1)/2}(s,t),d^h_{(i+1)/2}(s,t)]$ if $i$ is odd. 
		In the specific case when $w(\bar s,\bar t)=\bar t-\bar s$, we have $d^h_i=s+i2^{-h}(t-s)$ for every $i=0,\ldots,2^{h+1}$, which are the standard dyadic points of $[s,t]$.
		Thus,  the collection $D^h_w(s,t):=\{d^h_i(s,t)\}_{i=0}^{2^h}$ is regarded as $h$-dyadic points of $[s,t]$ with respect to the control $w$. As expected, these collections of points have similar properties to the standard dyadic points.
		Namely,  
		for every integers $h\ge0$ and $i=0,\dots,2^h-1$, we have $D^h_w(s,t)\subset D^{h+1}_w(s,t)$,
		\begin{equation}
		 	[d^h_i(s,t),d^h_{i+1}(s,t)]=[d^{h+1}_{2i}(s,t),d^{h+1}_{2i+1}(s,t)]\cup[d^{h+1}_{2i+1}(s,t),d^{h+1}_{2i+2}(s,t)]
		\end{equation} 
		 and
		\begin{equation}\label{est.wdh}
			w(d^{h}_{i}(s,t),d^{h}_{i+1}(s,t))\le 2^{-h} w(s,t).
		\end{equation}
		\begin{lemma}\label{lem.allocation}
			Let $A:\Omega\times \Delta\to V$ be $\cff\otimes\sbb(\Delta)/\sbb(V)$-measurable such that $A_{s,s}=0$ for every $s\in[0,T]$.
			Let $(s,t)$ be in $\Delta$. 
			Then, for every $N\ge0$ and every  $s\le t_0<\cdots <t_N\le t$, there exist a positive integer $h_0$ and random variables $R^h_i$, $i=0,\cdots,2^h-1$, $h\ge0$, such that
			\begin{enumerate}[wide,itemsep=2pt,labelindent=0pt,label={\upshape(\roman*)}]
				\item\label{p.i} $R^h_i=0$ for every $h\ge h_0$ and every $i$;
				\item\label{p.ii} for each $h,i$, there exist four (not necessarily distinct) points $s^{h,i}_1\le s^{h,i}_2\le s^{h,i}_3\le s^{h,i}_4$ in $[d^h_i(t_0,t_N),d^h_{i+1}(t_0,t_N)]$ so that
				\begin{equation}\label{form.Rni}
					R^h_i= A_{s^{h,i}_1,s^{h,i}_2}+A_{s^{h,i}_2,s^{h,i}_3}+A_{s^{h,i}_3,s^{h,i}_4}-A_{s^{h,i}_1,s^{h,i}_4}\,;
				\end{equation}
				\item\label{p.iii} the following identity holds 
				\begin{equation}\label{id.yaskov}
					\sum_{i=0}^{N-1} A_{t_i,t_{i+1}}-A_{t_0,t_N}=\sum_{h\ge0}\sum_{i=0}^{2^h-1}R^h_i.
				\end{equation}
			\end{enumerate}		
		\end{lemma}
		\begin{proof}
			For brevity, we abbreviate $d^h_i$ for  $d^h_i(t_0,t_N)$.
			For each collection $\cpp=\{s_i\}_{i=0}^K$ we define
			\[
				I(\cpp)=\sum_{i=0}^{K-1} A_{s_i,s_{i+1}}-A_{s_0,s_K}
				\quad\textmd{if}\quad K\ge1
			\] 
			and $I(\cpp)=0$ whenever $K=0$ or $\cpp$ is empty.
			For any two finite collections $\cpp_1,\cpp_2$, define
			\begin{equation}\label{def.dI}
				\delta I(\cpp_1,\cpp_2)=I(\cpp_1\cup \cpp_2)-I(\cpp_1)-I(\cpp_2).
			\end{equation}
			Put $\cpp^0_0=\{t_i\}_{i=0}^N$, which is a subset of $[d^0_0,d^0_1]$. The main idea of the proof is to allocate the elements of $\cpp^0_0$ into the $w$-dyadic subintervals of $[s,t]$ while keeping track of the resulting changes in $I(\cpp^0_0)$ during the process.
			For each $h\ge1$, define
			\[
				\cpp^h_{2^h-1}=\cpp^0_0\cap[d^h_{2^h-1},d^h_{2^h}]
				\quad\textrm{and}\quad
				\cpp^h_i=\cpp^0_0\cap[d^h_i,d^h_{i+1})
				\quad\textmd{for}\quad i=0,\dots,2^h-2.
			\]
			For each $n\ge0$ and $i=0,\dots,2^h-1$, define
			\begin{equation}\label{def.Rni}
				R^h_i:=\delta I(\cpp^{h+1}_{2i},\cpp^{h+1}_{2i+1})= I(\cpp^h_i)-I(\cpp^{h+1}_{2i})-I(\cpp^{h+1}_{2i+1})
			\end{equation}
			where the second identity comes from the fact that $\cpp^h_i=\cpp^{h+1}_{2i}\cup \cpp^{h+1}_{2i+1}$. We verify that the random variables $\{R^h_i\}_{h,i}$ satisfy \ref{p.i}-\ref{p.iii}.

			Since $\cpp^0_0$ is a finite set, there exists a finite integer $h_0\ge1$ so that $[d^h_i,d^h_{i+1}]\cap \cpp^0_0$ contains at most one point for every $h\ge h_0$ and every $i=0,\dots,2^{h}-1$. Hence, when $h\ge h_0$, we have $I(\cpp^{h}_i)=0$ and $R^h_i=0$ for every $i$. This shows \ref{p.i}. 

			If either $\cpp^{h+1}_{2i}$ or $\cpp^{h+1}_{2i+1}$ is empty, then $R^h_i=0$ and \eqref{form.Rni} is satisfied with $s^{h,i}_j=d^h_i$ for $j=1,2,3,4$. We assume that $\cpp^{h+1}_{2i}$ and $\cpp^{h+1}_{2i+1}$ are not empty. In such case, we define 
			\[
				s^{h,i}_1=\min \cpp^{h+1}_{2i},\, s^{h,i}_2=\max \cpp^{h+1}_{2i},\, s^{h,i}_3=\min \cpp^{h+1}_{2i+1}\, \mbox{ and } \,s^{h,i}_4=\max \cpp^{h+1}_{2i+1}.
			\] Here, $\min$ (respectively $\max$) of a nonempty finite set $F$ is the smallest (respectively largest) element of $F$.  We derive \eqref{form.Rni} from \eqref{def.Rni} and the definition of $I$ at the beginning of the proof.
			Hence, \ref{p.ii} is verified.

			Lastly, to show \ref{p.iii}, we apply \eqref{def.Rni} recursively to see that
			\begin{align*}
				I(\cpp^0_0)
				&=I(\cpp^1_0)+I(\cpp^1_1)+R^0_0
				\\&=\sum_{i=0}^{2^h-1}I(\cpp^h_i)+\sum_{k=0}^{h-1}\sum_{i=0}^{2^k-1}R^k_i
				\quad\textrm{for every}\quad h\ge1.
			\end{align*}
			Since $I(\cpp^h_i)=0$ as soon as $h\ge h_0$, the previous identity implies \eqref{id.yaskov}. This completes the proof.
		\end{proof}

		\begin{lemma}\label{lem.ANA}
		  Let $A$ be the process in \cref{prop.VSSL}. Then for every $N\ge0$ and every $S\le t_0<\cdots<t_N\le T$, we have
		  \begin{equation}\label{est.AN1}
		    \left\|\E_{t_0}\left(\sum_{i=0}^{N-1}A_{t_i,t_{i+1}}-A_{t_0,t_N}\right)\right\|_{V;n}\le C_1 \Gamma_1w(t_0,t_N)^{1+\varepsilon_1}
		  \end{equation}
		  and
		  \begin{align}\label{est.AN2}
		    \Big\|\Big\|\sum_{i=0}^{N-1}A_{t_i,t_{i+1}}-A_{t_0,t_N}\Big|\cff_{t_0}\Big\|_{V;m}\Big\|_n
		    \le C_1 \Gamma_1w(t_0,t_N)^{1+\varepsilon_1}
		    + C_2 \Gamma_2w(t_0,t_N)^{\frac1\p+\varepsilon_2}
		  \end{align}
		  where $C_1=2(1- 2^{-\varepsilon_1})^{-1}$ and $C_2=2C_{m,\p,V}(1- 2^{-\varepsilon_2})^{-1}$.
		\end{lemma}
		\begin{proof}
		  Put $s=t_0$ and $t=t_N$ and $d^h_i=d^h_i(s,t)$. Applying \cref{lem.allocation}, we can find random variables $R^h_i$, $i=0,\cdots, 2^h-1$, $h\ge0$ which satisfy properties \ref{p.i}-\ref{p.iii} stated there. 

		  Let $h\ge0$ be fixed. Define $\cgg^h_{2^h}=\cff_t$
		  and $\cgg_i^h=\cff_{s_1^{h,i}}$ for each $i=0,\dots,2^h-1$. We recall that $s^{h,i}_1$ is defined in \ref{p.ii}. The sequence $\{\cgg_i^h\}_{i=0}^{2^h}$ forms a filtration such that $R^h_i$ is $\cgg^h_{i+1}$-measurable for every $i=0,\cdots,2^h-1$.
		  The formula \eqref{form.Rni} can be written as
		  \[
		    R^h_i=-\delta A_{s^{h,i}_1,s^{h,i}_2,s^{h,i}_3}-\delta A_{s^{h,i}_1,s^{h,i}_3,s^{h,i}_4}.
		  \]
		  Applying the conditions \eqref{con.dA1} and \eqref{con.dA2}, we obtain from the previous identity that
		  \begin{equation}\label{tmp.rni}
		   	\|\E(R^h_i|\cgg^h_i)\|_{V;n}\le 2 \Gamma_1w(d^h_i,d^h_{i+1})^{1+\varepsilon_1}
		  \end{equation}
		  and
		  \begin{equation}\label{tmp.rni2}
		   	\|\|R^h_i-\E(R^h_i|\cgg^h_i)|\cgg^h_i\|_{V;m}\|_n\le 2 \Gamma_2w(d^h_i,d^h_{i+1})^{\frac12+\varepsilon_2}.
		  \end{equation}
		  From \eqref{ineq.DBDB.center}, we have
		  \begin{multline*}
		  	\|\|\sum_{i=0}^{2^h-1}R^h_i|\cff_{t_0}\|_{V;m}\|_n\le\sum_{i=0}^{2^h-1}\|\|\E(R^h_i|\cgg^h_i)|\cff_{t_0}\|_{V;m}\|_n
		  	\\+C_{m,\p,V}\left(\sum_{i=0}^{2^h-1}\|\|R^h_i-\E(R^h_i|\cgg^h_i)|\cff_{t_0}\|_{V;m}\|_n^\p\right)^{\frac1\p}.
		  \end{multline*}		 
		  Since $m\le n$ and $\cff_{t_0}\subset\cgg^h_i$, we have
		  \begin{align*}
		  	\|\|\E(R^h_i|\cgg^h_i)|\cff_{t_0}\|_{V;m}\|_n\le\E(R^h_i|\cgg^h_i)\|_{V;n}
		  \end{align*}
		  and
		  \[
		  	\|\|R^h_i-\E(R^h_i|\cgg^h_i)|\cff_{t_0}\|_{V;m}\|_n\le \|\|R^h_i-\E(R^h_i|\cgg^h_i)|\cgg^h_i\|_{V;m}\|_n.
		  \]
		  Taking into account \eqref{tmp.rni} and \eqref{tmp.rni2}, we get
		  \begin{multline*}
		  	\|\|\sum_{i=0}^{2^h-1}R^h_i|\cff_{t_0}\|_{V;m}\|_n\le2 \Gamma_1\sum_{i=0}^{2^h-1} w(d^h_i,d^h_{i+1})^{1+\varepsilon_1} 
		  	\\+2C_{m,\p,V} \Gamma_2\left(\sum_{i=0}^{2^h-1} w(d^h_i,d^h_{i+1})^{1+\p\varepsilon_2} \right)^{\frac1\p}.
		  \end{multline*}
		  From the estimate \eqref{est.wdh}, we see that 
		  \[
		  	\sum_{i=0}^{2^h-1} w(d^h_i,d^h_{i+1})^{1+\varepsilon_1} \le \sum_{i=0}^{2^h-1}2^{-h(1+\varepsilon_1)}w(s,t)^{1+\varepsilon_1}=2^{-h \varepsilon_1}w(s,t)^{1+\varepsilon_1}
		  \]
		  and similarly
		  \[
		  	\sum_{i=0}^{2^h-1} w(d^h_i,d^h_{i+1})^{1+\p\varepsilon_2}\le 2^{-h \p\varepsilon_2}w(s,t)^{1+\p\varepsilon_2}.
		  \]
		  We combine the previous inequalities to obtain that
		  \begin{equation}
		  	\|\|\sum_{i=0}^{2^h-1}R^h_i|\cff_{t_0}\|_{V;m}\|_n\le 2^{-h \varepsilon_1} 2 \Gamma_1w(s,t)^{1+\varepsilon_1}+2^{-h \varepsilon_2}2C_{m,\p,V}\Gamma_2 w(s,t)^{\frac1\p+\varepsilon_2}.
		  \end{equation}
		  From \eqref{id.yaskov}, applying triangle inequality and the above estimate, we see that
		  \begin{align*}
		  	\|\|\sum_{i=0}^{N-1}A_{t_i,t_{i+1}}-A_{t_0,t_N}|\cff_{t_0}\|_{V;m}\|_n
		  	&\le\left(\sum_{h\ge0}2^{-h \varepsilon_1}\right) 2 \Gamma_1w(s,t)^{1+\varepsilon_1}
		  	\\&\quad+\left(\sum_{h\ge0}2^{-h \varepsilon_2}\right) 2C_{m,\p,V}\Gamma_2 w(s,t)^{\frac1\p+\varepsilon_2}.
		  \end{align*}
		  This implies  \eqref{est.AN2}.
		  To show \eqref{est.AN1}, we obtain from \eqref{id.yaskov} and triangle inequality that
		   \begin{align*}
		  	\left\|\E\left(\sum_{i=0}^{N-1}A_{t_i,t_{i+1}}-A_{t_0,t_N}\Big|\cff_{t_0}\right)\right\|_{V;n}
		  	\le \sum_{h\ge0}\sum_{i=0}^{2^h-1}\|\E(R^h_i|\cff_{t_0})\|_{V;n}.
		  \end{align*}
		  We note that $\|\E(R^h_i|\cff_{t_0})\|_{V;n}\le \|\E(R^h_i|\cgg^h_i)\|_{V;n}$. Hence we can use \eqref{tmp.rni} to estimate the above sum analogously as before. This implies \eqref{est.AN1}.		
		\end{proof}	
		\begin{proof}[Proof of \cref{prop.VSSL}] We divide the proof into two steps. In the first step, a process $\caa$ is constructed as the limit of some Riemann sums. The second step verifies the properties \ref{cl:a} and \ref{cl:b}.

			\textit{Step 1.} Let $t$ be fixed but arbitrary in $[0,T]$. We show that the Riemann sums
			\[
				A_t^\pi=\sum_{i=0}^{n-1} A_{t_i,t_{i+1}}
			\]
			over partitions $\pi=\{t_i\}_{i=0}^n$ of $[0,t]$ has a limit in $L_m(V)$, denoted by $\caa_t$, as the mesh size $|\pi|_w:=\max_iw(t_i,t_{i+1})$ shrinks to $0$. 
			Let $\pi'=\{s_i\}_{i=0}^{n'}$ be another partition of $[0,t]$ and define $\pi''=\pi\cup \pi'$. We denote the points in $\pi''$ by $\{u_i\}_{i=0}^{n''}$ with $u_0\le u_1\le\cdots\le u_{n''}$ and some integer $n''\le n+n'$.  Then we have
			\begin{align*}
				A_t^{\pi''}-A_t^\pi=\sum_{i=0}^{n-1}Z_i
				\quad\textrm{where}\quad
				Z_i=\sum_{j:t_i\le u_j<t_{i+1}}A_{u_j,u_{j+1}}-A_{t_i,t_{i+1}}.
			\end{align*}
			Applying \eqref{ineq.DBDB}, we have
			\begin{align}\label{tmp.EZZ}
				\|A_t^{\pi''}-A_t^\pi\|_{V;m} 
				&\lesssim \sum_{i=0}^{n-1}\|\E_{{t_i}}Z_i\|_{V;m}+\left(\sum_{i=0}^{n-1}\|Z_i\|_{V;m}^\p\right)^{\frac1\p}.
			\end{align}
			Lemma \ref{lem.ANA} is applied to obtain that
			\begin{align*}
				\|\E_{{t_i}}Z_i\|_{V;m}\lesssim w(t_i,t_{i+1})^{1+\varepsilon_1}
				\quad\textrm{and}\quad
				\|Z_i\|_{V;m}\lesssim w(t_i,t_{i+1})^{\frac1\p+\varepsilon_2}.
			\end{align*}
			By super-additivity of $w$, this implies that
			\[
				\sum_{i=0}^{n-1}\|\E_{{t_i}}Z_i\|_{V;m}\lesssim \sum_{i=0}^{n-1}w(t_i,t_{i+1})^{1+\varepsilon_1}\lesssim  |\pi|_w ^{\varepsilon_1}
			\]
			and similarly $\sum_{i=0}^{n-1}\|Z_i\|_{V;m}^\p\lesssim |\pi|_w^{\p\varepsilon_2}$.
			Hence, we have
			\begin{align*}
				\|A_t^\pi-A_t^{\pi''}\|_{V;m} 
				&\lesssim |\pi|_w^{\varepsilon_1}+|\pi|_w^{\varepsilon_2}.
			\end{align*}
			The same argument is applied to $A_t^{\pi'}-A_t^{\pi''}$ which gives $\|A_t^{\pi'}-A_t^{\pi''}\|_{V;m}\lesssim |\pi'|_w^{\varepsilon_1}+|\pi'|_w^{\varepsilon_2}$. Hence, by triangle inequality, 
			\begin{align*}
				\|A_t^\pi-A_t^{\pi'}\|_{V;m}\le\|A_t^\pi-A_t^{\pi''}\|_{V;m}+\|A_t^{\pi'}-A_t^{\pi''}\|_{V;m}
				\lesssim |\pi|_w^{\varepsilon_1}+|\pi|_w^{\varepsilon_2}+|\pi'|_w^{\varepsilon_1}+|\pi'|_w^{\varepsilon_2}.
			\end{align*}
			This implies that $\{A_t^\pi-A_{0,t}\}_\pi$ is Cauchy in $L_m(V)$ and hence $\caa_t:=\lim_{|\pi|\to0}A_t^\pi$ is well-defined in probability.

			\textit{Step 2.} We show that the process $(\caa_t)_{0\le t\le T}$ defined in the previous step satisfies \ref{cl:a} and \ref{cl:b}.

			The condition \eqref{con.dA2} with $s=u=t$ implies that $A_{s,s}=0$ for every $s\in[0,T]$. Hence, it is evident that $\caa_0=0$. The fact that $A$ is $\{\cff_t\}$-adapted implies that $\caa$ is $\{\cff_t\}$-adapted. Obviously, $\caa_t-A_{0,t}$, being a limit in $L_m(V)$, belongs to $L_m(V)$ for each $t$. This shows \ref{cl:a}.

			Let $(s,t)$ be fixed but arbitrary in $\Delta$. Let $\pi=\{s=t_0<\cdots<t_N=t\}$ be an arbitrary partition of $[s,t]$.
			From construction of $\caa$ in the previous step, we see that
			\begin{equation}\label{tmp.aa}
				\caa_t-\caa_s-A_{s,t}=\lim_{|\pi|_w\downarrow 0}(A^\pi_{s,t} -A_{s,t})\quad\textrm{in}\quad L_m(V).
			\end{equation}
			Hence, passing through the limit $|\pi|_w\downarrow0$ in \eqref{est.AN1} and \eqref{est.AN2}, we obtain \eqref{est.A1} and \eqref{est.A2} respectively.
		\end{proof}
		\begin{proof}[Proof of \cref{prop.rate}]
			We start from a trivial identity
			\[
				\delta\caa_{s,t}-A^\pi=\sum_{[u,v]\in\pi}\left(\delta\caa_{u,v}-A_{u,v}\right).
			\]
			Then by \eqref{ineq.DBDB}, we have
			\begin{multline*}
				\|\delta\caa_{s,t}-A^\pi\|_{V;m}\le \sum_{[u,v]\in\pi}\|\E_u(\delta\caa_{u,v}-A_{u,v})\|_{V;m}
				\\+2C_{m,\p,V}\left(\sum_{[u,v]\in\pi}\|\delta\caa_{u,v}- A_{u,v}\|_{V;m}^\p\right)^{1/\p}.
			\end{multline*}
			Using the estimate \eqref{est.A1}, the first sum on the right-hand side above is bounded above by a constant multiple of
			\begin{align*}
				\Gamma_1\sum_{[u,v]\in\pi}w(u,v) ^{1+\varepsilon_1}\le \Gamma_1|\pi|_w^{\varepsilon_1}w(s,t)\,.
			\end{align*}
			The later sum is estimated similarly, using \eqref{est.A2}. This completes the proof.
		\end{proof}
		\begin{proof}[Proof of \cref{thm.DM}]
			For every $(s,t)\in \Delta$, we define $J_{s,t}=\E_sA_{s,t}$ and $M_{s,t}=A_{s,t}-\E_sA_{s,t}$. 
			Then for every $(s,u,t)\in \Delta_2$, we have
			\begin{align*}
				\delta J_{s,u,t}=\delta A_{s,u,t}+(\E_s-\E_u)A_{u,t}
				\quad,\quad \E_s \delta J_{s,u,t}=\E_s \delta A_{s,u,t}
			\end{align*}
			and
			\begin{align*}
				\delta M_{s,u,t}=-(\E_s-\E_u)A_{u,t}
				\quad,\quad \E_s \delta M_{s,u,t}=0\,.
			\end{align*}
			Applying \cref{prop.VSSL} for $J$ and $M$, we obtain the existence of the processes $\cjj$ and $\cmm$ respectively. Since $A$ is integrable, so are $\cjj$ and $\cmm$. The estimate \eqref{est.A1} yields \eqref{est:J}.
			In addition, the estimate \eqref{est.A2} implies that $\cjj$ satisfies \eqref{est:J'} and $\cmm$ is a martingale. Since $A_{s,t}=J_{s,t}+M_{s,t}$, it is evident that $\caa=\cjj+\cmm$. Uniqueness follows from \cref{lem.uniq}.
		\end{proof}
		
\section{Additive functionals of fractional Brownian motion} 
\label{sec:additive_functionals_of_fractional_brownian_motion}
	Let $B=(B^1,\dots,B^d)$ be a fractional Brownian motion in $\Rd$ with Hurst parameter $H\in(0,1)^d$. 
	Let $f$ be a time-dependent distribution in $L^\theta([0,T];\bes^\alpha_{p,\infty}(\Rd))$ where $\theta,p, \alpha$ are some fixed parameters, $p,\theta\in(1,\infty)$ and $\alpha\in\R$.
	In the current section, we study the additive functional
	\begin{align*}
		(t,x)\mapsto \int_0^t f_r(B_r+x)dr
	\end{align*}
	as the continuous extension of the map
	\begin{align*}
		\css(\Rd)\ni f\mapsto \left((t,x)\mapsto \int_0^t f_r(B_r+x)dr\right).
	\end{align*}

	For each $i\in\{1,\dots,d\}$, $B^i$ has Mandelbrot--Van Ness representation (\cite{MR242239}) with respect to a standard two sided Wiener process $W^i$ on $\R$, namely
	\begin{equation}\label{id.Mandelbrot}
		B^i_u=\int_{-\infty}^u[(u-r)_+^{H-\frac12}-(-r)_+^{H-\frac12}]dW^i_r
	\end{equation}
	where $(x)_+=\max\{x,0\}$. We assume that $W^i$'s (and hence $B^i$'s) are independent.
	For every $i\in\{1,\dots,d\}$ and $0\le u\le v$, we have
	\begin{equation}\label{def.fbm.incr}
		B^i_v=\int_{-\infty}^u[(v-r)^{H-\frac12}-(-r)_+^{H-\frac12}]dW^i_r+\int_u^v(v-r)^{H-\frac12}dW^i_r\,.
	\end{equation}
	which also yields that $\E_uB^i_v=\int_{-\infty}^u[(v-r)^{H-\frac12}-(-r)_+^{H-\frac12}]dW^i_r$.
	We define 
	\begin{equation}\label{def.rho}
		\rho(u,v)=\E\left(\int_u^v(v-r)^{H-\frac12}dW^i_r\right)^2=\frac1{2H}(v-u)^{2H}.
	\end{equation}

	For a $d\times d$-symmetric positive definite matrix $\Sigma$, $p_\Sigma$ denotes the density of a normal random variable in $\Rd$ with mean $0$ and variance $\Sigma$, i.e.
		\[
			p_\Sigma(x)=(2 \pi)^{-\frac d2} (\det \Sigma)^{\frac12}\exp\big(-\frac12 x^* \Sigma^{-1}x\big).
		\] 
		$P_\Sigma$ denotes the spatial convolution operator  with $p_\Sigma$. The  $d\times d$ identity matrix is denoted by $I$.

		\begin{proposition}\label{prop.aprif}
			Let $p,q$ be fixed numbers in $(1,\infty)$. Let $f$ be a function in $L^\theta([0,T];\css(\Rd))$, $\theta\in(1,\infty)$. 
			Let $\gamma$ be a positive number satisfying
			\begin{equation}\label{con.fbm.tbe}
				\gamma<\frac1{\hs}\left(1-\frac1{\min(2,\theta,p,q)}\right).
			\end{equation}
			Then for any $\alpha\in\R$, $m\ge2$ and $(s,t)\in \Delta$,
			\begin{equation}\label{est.f.besov}
				\Big\|\Big\|\int_s^t f_r(B_r+\cdot)dr|\cff_s\Big\|_{\bes^{\alpha+\gamma}_{p,q};m}\Big\|_\infty
				\le C|\1_{[s,t]}f|_{L^\theta\bes^\alpha_{p,\infty}} |t-s|^{1-\hs \gamma-\frac1 \theta}.
			\end{equation}
		\end{proposition}
		\begin{proof}
			Define $A_{s,t}(x)=\int_s^tP_{\rho(s,r)I}f_r(\E_sB_r+x)dr$. We note that the Besov norm $|\cdot|_{\bes^\alpha_{p,q}}$ is translation invariant, i.e. $|f(y+\cdot)|_{\bes^{\alpha+\gamma}_{p,q}}=|f|_{\bes^{\alpha+\gamma}_{p,q}}$ for every $y\in\Rd$. 
			Note that 
			\begin{align*}
				\wei{\rho(s,r)I\xi,\xi}\ge \frac12(r-s)^{2\hs}|\xi|^2.
			\end{align*}
			Using triangle inequality and \cref{lem.Pag}, we see that 
			\begin{align*}
				|A_{s,t}|_{\bes^{\alpha+\gamma}_{p,q}}&\le\int_s^t|P_{\rho(s,r)I}f_r(\E_sB_r+\cdot)|_{\bes^{\alpha+\gamma}_{p,q}}dr
				\\&=\int_s^t|P_{\rho(s,r)I}f_r|_{\bes^{\alpha+\gamma}_{p,q}}dr
				\lesssim\int_s^t |r-s|^{-\hs \gamma}|f_r|_{\bes^\alpha_{p,\infty}} dr.
			\end{align*}
			By H\"older inequality, we have
			\begin{align*}
				|A_{s,t}|_{\bes^{\alpha+\gamma}_{p,q}}\lesssim |\1_{[s,t]}f|_{L^\theta\bes^\alpha_{p,\infty}}|t-s|^{1-\hs \gamma-\frac1 \theta},
			\end{align*}
			where we have used \eqref{con.fbm.tbe} to ensure that the integral in time is finite. Define the continuous control $w$ by
			\begin{equation}\label{def.w.bes}
				|w(s,t)|^{1-\hs \gamma}=|\1_{[s,t]}f|_{L^\theta\bes^\alpha_{p,\infty}}|t-s|^{1-\hs \gamma-\frac1 \theta},
			\end{equation}
			so that the previous estimate yields
			\begin{align}\label{tmp.Ast111}
				|A_{s,t}|_{\bes^{\alpha+\gamma}_{p,q}}\les w(s,t)^{1-H \gamma}.
			\end{align}
			We recall from \cref{prop.ex.type} that $\bes^{\alpha+\gamma}_{p,q}$ has martingale type $\p=\min\{2,p,q\}$.		
			The condition \eqref{con.fbm.tbe} also implies that $1-\hs \gamma>\frac1\p$.
			In view of \eqref{tmp.Ast111}, condition \eqref{con.dA2} is satisfied. It is straightforward to verify that $\E_s \delta A_{s,u,t}=0$ for every $s\le u\le t$, hence condition \eqref{con.dA1} holds trivially.
			Hence, we can apply the stochastic sewing lemma, \cref{prop.VSSL}, to define the process $(\caa_t)$ as in \eqref{def.caa}. The estimates \eqref{est.A2} and \eqref{tmp.Ast111} imply that
			\begin{align*}
				\|\|\delta\caa_{s,t}|\cff_s\|_{\bes^{\alpha+\gamma}_{p,q};m}\|_\infty\les |1_{[s,t]}f|_{L^\theta\bes^\alpha_{p,\infty}}|t-s|^{1-H \gamma-\frac1 \theta}.
			\end{align*}
			Hence, to obtain \eqref{est.f.besov}, it remains to show that $\caa_t=\int_0^t f_r(B_r+\cdot)dr$.

			We put
			\begin{align*}
				R_{s,t}(x)=\int_s^tf_r(B_r+x)dr-A_{s,t}(x).
			\end{align*}
			It is evident that $\E_s R_{s,t}(x)=0$. In addition, by Minkowski inequality
			\begin{align*}
				\|\int_s^t f_r(B_r+\cdot)dr\|_{\bes^{\alpha+\gamma}_{p,q}}\le\int_s^t\|f_r\|_{\bes^{\alpha+\gamma}_{p,q}}dr.
			\end{align*}
			Combining with \eqref{tmp.Ast111} gives
			\begin{align*}
				|R_{s,t}|_{\bes^{\alpha+\gamma}_{p,q}}\les w(s,t)^{1-H \gamma}+\int_s^t\|f_r\|_{\bes^{\alpha+\gamma}_{p,q}}dr.
			\end{align*}
			Applying \cref{lem.uniq}, we have that for each $t\in[0,T]$,
			\begin{align*}
				\int_0^t f_r(B_r+\cdot)dr=\lim_{\pi\in\cpp([0,t]), |\pi|\to0}\sum_{[u,v]\in \pi}A_{u,v},
			\end{align*}
			which shows that $\caa_t=\int_0^t f_r(B_r+\cdot)dr$ and hence, completes the proof.
		\end{proof}
		For each $f\in L^\theta([0,T];\css)$, define
		\begin{align}\label{def.IfB}
			I[f]_t(x)=\int_0^tf_r(B_r+x)dr.
		\end{align}
		\cref{prop.aprif} shows that the map
		\begin{align*}
			I:L^\theta([0,T];\css\cap\bes^\alpha_{p,\infty})\to C([0,T];L_m\bes^{\alpha+\gamma}_{p,q})
		\end{align*}
		is bounded. 
		By a density argument, we can extend $I$ to $\bes^\alpha_{p,1}$.
		\begin{theorem}\label{thm.fbm.besov}
			For any $\theta,p,q\in(1,\infty)$ and $\gamma$ satisfying \eqref{con.fbm.tbe}, the map 
			\begin{align*}
				I:L^\theta([0,T];\bes^\alpha_{p,1}(\Rd))\to C([0,T];L_m\bes^{\alpha+\gamma}_{p,q}(\Rd))
			\end{align*}
			is well-defined as the unique continuous extension of \eqref{def.IfB}. In addition, one has for any $f\in\bes^\alpha_{p,1}(\Rd)$
			\begin{align}\label{est.Icont}
				\|\|\delta I[f]_{s,t}|\cff_s\|_{\bes^{\alpha+\gamma}_{p,q};m}\|_\infty\le C|\1_{[s,t]}f|_{L^\theta\bes^\alpha_{p,\infty}}	|t-s|^{1-H \gamma-\frac 1 \theta}.	
			\end{align}						
		\end{theorem}
		\begin{proof}
			Using the embedding $\bes^\alpha_{p,1}\hookrightarrow\bes^\alpha_{p,\infty}$ and the trivial estimate $\|\delta I[f]_{s,t}\|_{\bes^{\alpha+\gamma}_{p,q};m}\le\|\|\delta I[f]_{s,t}|\cff_s\|_{\bes^{\alpha+\gamma}_{p,q};m}\|_\infty$,  we obtain from \eqref{est.f.besov} that
			\[
				\|\delta I[f]_{s,t}\|_{\bes^{\alpha+\gamma}_{p,q};m}\le C|\1_{[s,r]}f|_{L^\theta([0,T];\bes^\alpha_{p,1})}|t-s|^{1-\hs \gamma-\frac1 \theta}.
			\]
			for every $f\in\css$. Since $\css$ is dense in $\bes^\alpha_{p,1}$ and the action $f\mapsto I[f]$ is linear, this implies that $I$ has a unique continuous extension $I:L^\theta([0,T];\bes^{\alpha}_{p,1})\to C([0,T];L_m\bes^{\alpha+\gamma}_{p,q})$.

			Moving on to \eqref{est.Icont}. Let $f$ be in $L^\theta([0,T];\bes^\alpha_{p,1})$. For each $n\ge1$, define $f^n=P_{1/nI}f$ which belongs to $L^\theta([0,T];\css)$. We have $\lim_n f_n=f$ in $L^\theta([0,T];\bes^\alpha_{p,1})$ and $|f^n_r|_{\bes^\alpha_{p,\infty}}\le|f_r|_{\bes^\alpha_{p,\infty}}\le|f_r|_{\bes^\alpha_{p,1}}$.
			Then from \eqref{est.f.besov}, we get that
			\begin{align*}
				\|\|\delta I[f^n]_{s,t}|\cff_s\|_{\bes^{\alpha+\gamma}_{p,q};m}\|_\infty
				&\le C|\1_{[s,t]}f^n|_{L^\theta\bes^\alpha_{p,\infty}} |t-s|^{1-\hs \gamma-\frac1 \theta}
				\\&\le C|\1_{[s,t]}f|_{L^\theta\bes^\alpha_{p,\infty}} |t-s|^{1-\hs \gamma-\frac1 \theta}.
			\end{align*}
			Passing through the limit $n\to\infty$, using the continuity of $I$ on $L^\theta([0,T];\bes^\alpha_{p,1})$ we obtain \eqref{est.Icont}.
		\end{proof}
		\begin{remark}
			In the case $f\in L^\infty([0,T];\bes^\alpha_{p,1}(\Rd))$ the control $w$ defined by the relation \eqref{def.w.bes} is not necessary continuous. However, for $f\in C([0,T];\bes^\alpha_{p,1}(\Rd))$, the arguments of \cref{prop.aprif} and \cref{thm.fbm.besov} are still valid by simply setting $\theta=\infty$.
		\end{remark}
		\begin{remark}
			Using Besov-Sobolev embeddings
			\begin{gather*}
				\bes^{\alpha}_{p,1}\hookrightarrow W^\alpha_p\hookrightarrow \bes^\alpha_{p,\infty}\hookrightarrow\bes^{\alpha- \varepsilon}_{p,1} \quad\textrm{for}\quad \varepsilon>0,
				\\
				\bes^\alpha_{p,q}\hookrightarrow \bes^\beta_{p_1,q_1}
				\quad\textrm{for}\quad
				\alpha-\frac dp=\beta-\frac d{p_1},\quad p\le p_1, \quad q\le q_1,
			\end{gather*}
			and the isomorphism $\bes^\alpha_{p,p}=W^\alpha_p$ when $\alpha$ is not an integer,
			\cref{thm.fbm.besov} can be applied to distributions $f$ in $L^\theta([0,T];\bes^\alpha_{p,\infty})$ and $L^\theta([0,T];W^\alpha_{p})$.
		\end{remark}
		\begin{remark}
			In view of \cref{prop.kolmo} and \eqref{est.Icont}, for each $f\in\bes^\alpha_{p,1}(\R)$, $I[f]$ has a continuous modification (as a process taking values in $\bes^{\alpha+\gamma}_{p,q}(\Rd)$).
		\end{remark}
		
		\cref{prop.rate} provides an alternative approximation for $I[f]$ by Riemann sums.
		\begin{corollary}
			Let $p,q,\theta\in(1,\infty)$, $\alpha\in\R$ and $f$ be in $L^\theta([0,T];\bes^\alpha_{p,1})$.
			Let $(s,t)\in \Delta$ and $\pi$ be a partition of $[s,t]$ and define the Riemann sum
			\begin{align*}
				I^\pi[f]_{s,t}(x)=\sum_{[u,v]\in \pi}\int_u^vP_{\rho(u,r)I}f_r(\E_uB_r+x)dr.
			\end{align*}
			Let $w$ be the control defined by the relation \eqref{def.w.bes}.
			Then for any $\gamma$ satisfying \eqref{con.fbm.tbe},
			\begin{align}
				\|\delta I[f]_{s,t}-I^\pi[f]_{s,t}\|_{\bes^{\alpha+\gamma}_{p,q}(\Rd);m}\les|\pi|_w^{1-H \gamma-\frac1{\min(2,p,q)}}w(s,t)^{\frac1{\min(2,p,q)}}
			\end{align}
		\end{corollary}
		\begin{proof}
			Straightforward from \cref{thm.fbm.besov} and \cref{prop.rate}.
		\end{proof}

		Using embeddings between Besov spaces and H\"older--Zygmund spaces $\C^\beta$, we can derive from \cref{thm.fbm.besov} the almost sure continuity the additive functional $I[f]$.
		\begin{corollary}\label{cor.fc}
			Let $p,\theta\in(1,\infty)$, $\alpha\in\R$ and $f$ be in $L^\theta([0,T];\bes^\alpha_{p,1})$.
			\begin{enumerate}[(i)]
				\item Assume that
				\begin{equation}\label{con.fbm.holder}
					\alpha+\frac1{\hs}\left(1-\frac1{\min(2,\theta,p)}\right)>\frac dp.
				\end{equation}
				Then for every $\beta$ satisfying
				\begin{equation}\label{con.fbm.beta}
					0<\beta<\alpha-\frac dp+\frac1{\hs}\left(1-\frac1{\min(2,\theta,p)}\right),
				\end{equation}
				we have
				\begin{equation}\label{est.f.hder}
					\|\|\delta I[f]_{s,t}|\cff_s\|_{\C^\beta;m}\|_\infty\lesssim|\1_{[s,t]}f|_{L^\theta \bes^\alpha_{p,\infty}} |t-s|^{1-\hs (\beta- \alpha+\frac dp)-\frac1 \theta}
				\end{equation}
				for every $(s,t)\in \Delta$ and every $m\ge2$.
				\item Assume that 
				\begin{equation}\label{con.fbm.Lp}
					\alpha+\frac1{H}\left(1-\frac1{\min(2,\theta,p)}\right)>0.
				\end{equation}
				Then for any $v\in[p,\infty]$ satisfying
				\begin{align}\label{con.v}
					 \alpha+\frac1{H}\left(1-\frac1{\min(2,\theta,p)}\right)>\frac dp- \frac dv,
				 \end{align}
				there exists $\gamma=\gamma(v)$ satisfying \eqref{con.fbm.tbe} such that
				\begin{equation}\label{est.f.Lp}
					\|\|\delta I[f]_{s,t}|\cff_s\|_{L^v;m}\|_\infty\lesssim|\1_{[s,t]}f|_{L^\theta \bes^\alpha_{p,\infty}} |t-s|^{1-\hs \gamma -\frac1 \theta}
				\end{equation}
				for every $(s,t)\in \Delta$ and every $m\ge2$.
			\end{enumerate}
		\end{corollary}
		\begin{proof} 
			(i) Assume that $\beta$ satisfy \eqref{con.fbm.beta} and define $\gamma$ by the relation $\beta=\gamma+\alpha-\frac dp$. Then $\gamma$ satisfies the condition \eqref{con.fbm.tbe} and we have the embedding $\bes^{\alpha+\gamma}_{p,p}\hookrightarrow \C^\beta$. From \cref{thm.fbm.besov}, we deduce that $I[f]_t$ belongs to $\C^\beta$ and the estimate \eqref{est.f.hder} follows from \eqref{est.f.besov}.

			(ii) From the conditions \eqref{con.fbm.Lp} and \eqref{con.v}, we can choose $\gamma$ satisfying  \eqref{con.fbm.tbe} (with $q=p$)  such that $\gamma+\alpha>\frac dp-\frac dv$. It suffices to apply \cref{thm.fbm.besov} and the embeddings $\bes^{\gamma+\alpha}_{p,p}\hookrightarrow\bes^\varepsilon_{v,v}\hookrightarrow L^v(\Rd)$, $\gamma+\alpha-\frac dp=\varepsilon-\frac dv$.
		\end{proof}
		\begin{remark}
			We do not require $\beta\in(0,1)$ in \cref{cor.fc}(i). This means that when $k:=\alpha-\frac dp+\frac1\hs\left(1-\frac1{\min(2,\theta,p)}\right)$ is larger than $1$, the functional $I[f]_t$ is $n$-times differentiable in the spatial variables for any integer $n<k$. 
		\end{remark}
		The class $ L^\theta([0,T];\bes^0_{1,\infty}(\Rd))$ contains Dirac distributions and the corresponding functional $I$ is directly related to the local time of fractional Brownian motion. 
		For this class,
		\cref{thm.fbm.besov} is still applicable through the Besov embedding $\bes^0_{1,\infty}\hookrightarrow\bes^{-d/2}_{2,2}$. 
		\begin{corollary}\label{cor.flp}
			Let $f$ be a distribution in $ L^\theta([0,T];\bes^0_{1,\infty}(\Rd))$ with $\theta\ge2$.
			
			(i) (Small $H$) When $0< Hd<\frac12$, for every $t\in[0,T]$,  $I[f]_t$ belongs to $L^u(\Rd)$ almost surely for every $v\in[2,\infty]$. 

			(ii) (Large $H$) When $\frac12\le Hd<1$, for every $t\in[0,T]$, $I[f]_t$ belongs to $L^v(\Rd)$ almost surely for every $v\in[2,\frac{2Hd}{2Hd-1})$. Here we use the convention that $\frac{2Hd}{2Hd-1}=\infty$ if $Hd=\frac12$.
		\end{corollary}
		\begin{proof}
			Straightforward from the embedding $\bes^0_{1,\infty}\hookrightarrow\bes^{-d/2}_{2,2}$ and \cref{cor.fc}. 	
		\end{proof}

\appendix
\section{A continuity criterion with controls} 
\label{sec:a_continuity_criterion_with_controls}
	We give an extension of the classical Kolmogorov continuity theorem with generic controls.
	\begin{proposition}\label{prop.kolmo}
		Let $V$ be a Banach space, $w$ be a continuous control which is $w$ strictly increasing, i.e. $w(u,v)<w(s,t)$ whenever $[u,v]\subsetneq[s,t]$. Let $m\ge1$ and $\alpha\in(0,1]$ be such that $\beta_0:=\alpha-1/m>0$. Let $\caa$ be a $V$-valued process such that
		\begin{align}\label{con.caama}
			\|\delta \caa_{s,t}\|_{V;m}\le w(s,t)^{\alpha} \quad \forall(s,t)\in \Delta.
		\end{align}
		Then $\caa$ has continuous modification $\tilde\caa$ and for every $\beta\in(0,\beta_0)$, there is a finite constant $C(\beta, \beta_0)$ such that
		\begin{align}\label{est.kol}
			\left\|\sup_{(s,t)\in \Delta,s<t} \frac{|\delta\tilde\caa_{s,t}|_V}{w(s,t)^\beta}\right\|_m\le C(\beta,\beta_0) w(0,T)^{\alpha- \beta}.
		\end{align}
	\end{proposition}
	\begin{proof}
		The proof is similar to the standard Kolmogorov continuity theorem (\cite[pg. 26]{MR1725357}) with some minor modifications to replace the standard dyadic points by the dyadic points with respect to the control $w$ (defined in \cref{sub:proofs}).

		Without loss of generality, we assume that $w(0,T)=1$.
		Recall the definition of $d^h_i(0,T)$ from \cref{sub:proofs}.
		Let $d^h_i=d^h_i(0,T)$, $D^h=\{d^h_i\}_{i=0}^{2^h-1}$ and $D=\cup_{h\ge0} D^h$. Because of the monotonicity of $w$, we  have $d^h_i<d^h_{i+1}$ and that $D$ is a dense subset of $[0,T]$.   Define
		\[
			K_h=\sup_{s,t\in D^{h}:s\le t, w(s,t)\le 2^{1-h}}|\caa_t-\caa_s|_V, \quad h\ge0,
		\]
		and recall that $\beta_0= \alpha-1/m>0$.
		Then by \eqref{con.caama},
		\begin{align*}
		 	\E |K_h|^m\le\sum_{s,t\in D^{h}: s\le t,w(s,t)\le 2^{1-h}}\E|\delta\caa_{s,t}|_V^m\le 2^{h+1}2^{(1-h) m \alpha}=2^{1+m \alpha} 2^{-h m \beta_0}.
		\end{align*}

		Let $s,t$ be in $D$, $s<t$. For each $n$, define
		\begin{align*}
			s_n=\inf\{r\in D^n:r\ge s\}
			\tand t_n=\sup\{r\in D^n:r\le t\}.
		\end{align*}
		It is straightforward to see that $(s_n)_n$ is decreasing and $s_n=s$ for some $n$ on; $(t_n)_n$ is increasing and $t_n=t$ for some $n$ on; and for every $n$
		\begin{align*}
			w(s_{n+1},s_n)\le2^{-n}, \quad w(t_n,t_{n+1})\le 2^{-n}.
		\end{align*}
		To see this, let $s_n'$ be the dyadic point adjacent to $s_n$ to the left, i.e. $s_n'=\max\{r\in D^n:r<s_n\}$. Then we have $s'_n\le s\le s_n$, so that $w(s_{n+1},s_n)\le w(s,s_n)\le w(s'_n,s_n)\le 2^{-n}$ by \eqref{est.wsut}. Similarly, let $t_n'$ be the dyadic point adjacent to $t_n$ to the right. One has $t_n\le t\le t_n'$, which implies the estimate for $w(t_n,t_{n+1})$.

		If $h\ge0$ is an integer satisfying $w(s,t)\le 2^{-h}$, then we have additionally that
		\begin{align*}
			\1_{(s_h\le t_h)} w(s_h,t_h)+\1_{(t_h\le s_h)} w(t_h,s_h)\le2^{-h}.
		\end{align*}
		Indeed, if $s_h\le t_h$, then $[s_h,t_h]\subset [s,t]$ and we have $w(s_h,t_h)\le w(s,t)\le 2^{-h}$.  If $t_h< s_h$ then one has $t_h\le s\le t\le s_h$. In this case, we must have the identity $t_h=s_h'$, where $s_h'$ is defined previously (namely, the dyadic point adjacent to $s_h$ to the left). Then $w(t_h,s_h)=w(s_h',s_h)\le 2^{-h}$ by \eqref{est.wsut}.

		Let $s,t$ be in $D$, $s<t$ and $w(s,t)\le 2^{-h}$. We have
		\begin{align*}
			\delta\caa_{s,t}=\delta\caa_{s_h,t_h}+ \sum_{i=h}^\infty \delta\caa_{s_i,s_{i+1}}+\sum_{i=h}^\infty \delta\caa_{t_i,t_{i+1}},
		\end{align*}
		where the series are actually finite sums. From the definition of $K_n$ and properties of $s_n,t_n$ described previously, it follows that
		\begin{align*}
			|\delta\caa_{s,t}|_V\le2 K_{h}+2\sum_{i=h}^\infty K_{i+1}\le 2\sum_{i=h}^\infty K_i.
		\end{align*}
		Consequently, setting $M_\beta=\sup\{|\delta\caa_{s,t}|_V/w(s,t)^\beta;s,t\in D,s< t\}$ for $\beta\in(0,\beta_0)$, we have
		\begin{align*}
			M_\beta&\le\sup_{h\ge0}\sup_{s,t\in D: s<t, w(s,t)\le 2^{-h}}2^{(h+1)\beta}|\delta\caa_{s,t}|_V
			\\&\le \sup_{h\ge0} 2^{1+(h+1)\beta}\sum_{i=h}^\infty K_i
			\le 2^{1+\beta}\sum_{i=0}^\infty 2^{i \beta} K_i.
		\end{align*}
		It follows that
		\begin{align*}
			\|M_\beta\|_m
			\le 2^{1+\beta}\sum_{i=0}^\infty2^{i \beta}\|K_i\|_{m}
			\le 2^{3+\beta}\sum_{i=0}^\infty2^{i(\beta- \beta_0)}<\infty.
		\end{align*}
		In particular, for a.e. $\omega$, $\caa_\cdot(\omega)$ is uniformly continuous on $D$ and it makes sense to define for every $t\in[0,T]$,
		\begin{align*}
		 	\tilde{\caa}_t(\omega)=\lim_{s\to t, s\in D}\caa_s(\omega).
		\end{align*}
		It is now standard to verify that $\tilde \caa$ is the desired modification and that \eqref{est.kol} holds with the constant $C(\beta,\beta_0)=2^{3+\beta}\sum_{i=0}^\infty2^{i(\beta- \beta_0)}$.
	\end{proof}

\section{Auxiliary estimates} 
\label{sec:besov_spaces}
		To obtain various properties of Besov spaces, we will make use of the following Bernstein's inequalities. 
		Let $f$ be a function in $L^p(\Rd)$ and let $q\ge p$, $p,q\in[1,\infty]$. For every integer $k\ge0$, every $\lambda>0$ and $t>0$ we have (\cite[Lemma 2.1]{bahouri})
		\begin{align}
			&\supp F f\subset \lambda\cbb \Rightarrow \|\nabla^k f\|_{L^q(\Rd)}\le C^{k+1}\lambda^{k+d(\frac1p-\frac1q)}\|f\|_{L^p(\Rd)},
			\label{est.bernB}
		\end{align}
		where $F f$ denotes the Fourier transform of $f$ and $\cbb$ is a ball centered at $0$ in $\Rd$.

		\begin{lemma}
			Let $\cnn$ be an annulus. Let $a$ be a $d\times d$-matrix such that
			\[
				\wei{a \xi,\xi}\ge \kappa_1|\xi|^2
			\] 
			for some $\kappa_1\in(0,M)$.
			Then there exist positive constants $c,C=C(d,M)$ such that for any $\lambda>0$, $p\in[1,\infty]$ and any function $g$ whose Fourier transform is supported in $\lambda\cnn$
			\begin{equation}\label{est.Bernheat}
				|P_a g|_{L^p}\le C e^{-c \kappa_1 \lambda^2} |g|_{L^p}.
			\end{equation}		
		\end{lemma}
		\begin{proof}
			\cite[Lemma 2.4]{bahouri}.
		\end{proof}
		\begin{lemma}\label{lem.Pag}
			For $\gamma>0$, $(p,q)\in[1,\infty]^2$,
			$|P_ag|_{\bes^{\alpha+\gamma}_{p,q}}\lesssim (1+\kappa^{-\gamma/2})|g|_{\bes^\alpha_{p,\infty}} $.
		\end{lemma}
		\begin{proof}		
			From the embedding $\bes^{\alpha+\gamma}_{p,1}\hookrightarrow\bes^{\alpha+\gamma}_{p,q}$, it suffices to consider the case $q=1$.
			We denote by $\Delta_j$, $j\ge-1$, the (nonhomogeneous) Littlewood-Paley blocks (\cite{bahouri}*{page 61}).
			We have from \eqref{est.Bernheat} and \eqref{est.bernB}, for $j\ge0$,
			\begin{align*}
				|P_a(\Delta_jg)|_{L^p}\les e^{-2^{2j} \kappa}|\Delta_j g|_{L^p}\les e^{-2^{2j} \kappa}2^{- \alpha j}|g|_{\bes^\alpha_{p,\infty}}.
			\end{align*}
			Noting that  $\Delta_j P_ag=P_a (\Delta_jg)$, we obtain
			\begin{align*}
				2^{j(\alpha+\gamma)}|\Delta_j(P_ag)|_{L^p}\les|g|_{\bes^\alpha_{p,\infty}}e^{-2^{2j} \kappa}2^{ \gamma j}.
			\end{align*}
			Since $\gamma>0$, it is easy to check that
			\begin{align*}
				\sup_{\kappa>0}\kappa^{\frac \gamma2} \sum_{j\ge0}e^{-2^{2j} \kappa}2^{\gamma j}<\infty
			\end{align*}
			which implies that
			\begin{align*}
				\sum_{j\ge0}|\Delta_j (P_ag)|_{L^p}\les |g|_{\bes^\alpha_{p,\infty}}\kappa^{-\frac \gamma2}.
			\end{align*}
			For $j=-1$, we have
			\begin{align*}
				|\Delta_{-1} (P_ag)|_{L^p}=|P_a(\Delta_{-1}g)|_{L^p}\le |\Delta_{-1}g|_{L^p}\les |g|_{\bes^\alpha_{p,\infty}}.
			\end{align*}
			Combining the previous two estimates, we obtain the result.		
		\end{proof}
	\paragraph{\bf Acknowledgment} 
	\label{par:acknowledgment}
		I thank Pavel Zorin-Kranich for various discussions about \cref{prop.ex.type}.
	\paragraph{\bf Funding statement} 
	\label{par:funding_statement_}
		The research is supported partly by the Alexander von Humboldt Foundation and the European Research Council under the European Union’s Horizon 2020 research and innovation program (grant agreement No. 683164).
		Supports from the Hausdorff Research Institute for Mathematics (HIM), where the research was initiated during the junior trimester program in 2019, are  acknowledged.
\bibliographystyle{plain}
\bibliography{../bibliography/processes}

\begin{bibdiv}
\begin{biblist}

\bib{ABLM}{article}{
      author={Athreya, Siva},
      author={Butkovsky, Oleg},
      author={L{\^e}, Khoa},
      author={Mytnik, Leonid},
       title={Well-posedness of stochastic heat equation with distributional
  drift and skew stochastic heat equation},
        date={2020},
     journal={arXiv preprint arXiv:2011.13498},
}

\bib{bahouri}{book}{
      author={Bahouri, Hajer},
      author={Chemin, Jean-Yves},
      author={Danchin, Rapha{\"e}l},
       title={Fourier analysis and nonlinear partial differential equations},
      series={Grundlehren der Mathematischen Wissenschaften [Fundamental
  Principles of Mathematical Sciences]},
   publisher={Springer},
     address={Heidelberg},
        date={2011},
      volume={343},
        ISBN={978-3-642-16829-1},
         url={http://dx.doi.org/10.1007/978-3-642-16830-7},
      review={\MR{2768550 (2011m:35004)}},
}

\bib{butkovsky2021approximation}{article}{
      author={Butkovsky, Oleg},
      author={Dareiotis, Konstantinos},
      author={Gerencs{\'e}r, M{\'a}t{\'e}},
       title={Approximation of {SDE}s: a stochastic sewing approach},
        date={2021},
     journal={Probability Theory and Related Fields},
       pages={1\ndash 60},
}

\bib{MR942266}{incollection}{
      author={Cobos, Fernando},
      author={Fernandez, Dicesar~Lass},
       title={Hardy-{S}obolev spaces and {B}esov spaces with a function
  parameter},
        date={1988},
   booktitle={Function spaces and applications ({L}und, 1986)},
      series={Lecture Notes in Math.},
      volume={1302},
   publisher={Springer, Berlin},
       pages={158\ndash 170},
         url={https://doi.org/10.1007/BFb0078872},
      review={\MR{942266}},
}

\bib{MR3505229}{article}{
      author={Catellier, R.},
      author={Gubinelli, M.},
       title={Averaging along irregular curves and regularisation of {ODE}s},
        date={2016},
        ISSN={0304-4149},
     journal={Stochastic Process. Appl.},
      volume={126},
      number={8},
       pages={2323\ndash 2366},
         url={https://doi.org/10.1016/j.spa.2016.02.002},
      review={\MR{3505229}},
}

\bib{MR835828}{article}{
      author={Cobos, Fernando},
       title={Clarkson's inequalities for {S}obolev spaces},
        date={1986},
        ISSN={0025-5513},
     journal={Math. Japon.},
      volume={31},
      number={1},
       pages={17\ndash 22},
      review={\MR{835828}},
}

\bib{dareiotis2021quantifying}{misc}{
      author={Dareiotis, Konstantinos},
      author={Gerencs{é}r, M{á}t{é}},
      author={L{ê}, Khoa},
       title={Quantifying a convergence theorem of {G}y{\"o}ngy and {K}rylov},
        date={2021},
        note={to appear},
}

\bib{MR2261056}{article}{
      author={Feyel, Denis},
      author={de~La~Pradelle, Arnaud},
       title={Curvilinear integrals along enriched paths},
        date={2006},
        ISSN={1083-6489},
     journal={Electron. J. Probab.},
      volume={11},
       pages={no. 34, 860\ndash 892},
         url={https://doi.org/10.1214/EJP.v11-356},
      review={\MR{2261056}},
}

\bib{MR4174393}{book}{
      author={Friz, Peter~K.},
      author={Hairer, Martin},
       title={A course on rough paths},
     edition={Second},
      series={Universitext},
   publisher={Springer, Cham},
        date={2020},
        ISBN={978-3-030-41556-3; 978-3-030-41555-6},
         url={https://doi.org/10.1007/978-3-030-41556-3},
        note={With an introduction to regularity structures},
      review={\MR{4174393}},
}

\bib{FHL21}{article}{
      author={Friz, Peter~K},
      author={Hocquet, Antoine},
      author={L{\^e}, Khoa},
       title={Rough stochastic differential equations},
        date={2021},
     journal={arXiv preprint arXiv:2106.10340},
}

\bib{g2020regularisation}{article}{
      author={Gerencs{\'e}r, M{\'a}t{\'e}},
       title={Regularisation by regular noise},
        date={2020},
     journal={arXiv preprint arXiv:2009.08418},
}

\bib{MR556414}{article}{
      author={Geman, Donald},
      author={Horowitz, Joseph},
       title={Occupation densities},
        date={1980},
        ISSN={0091-1798},
     journal={Ann. Probab.},
      volume={8},
      number={1},
       pages={1\ndash 67},
  url={http://links.jstor.org/sici?sici=0091-1798(198002)8:1<1:OD>2.0.CO;2-M&origin=MSN},
      review={\MR{556414}},
}

\bib{MR2091358}{article}{
      author={Gubinelli, M.},
       title={Controlling rough paths},
        date={2004},
        ISSN={0022-1236},
     journal={J. Funct. Anal.},
      volume={216},
      number={1},
       pages={86\ndash 140},
         url={https://doi.org/10.1016/j.jfa.2004.01.002},
      review={\MR{2091358}},
}

\bib{MR4124526}{article}{
      author={Hairer, Martin},
      author={Li, Xue-Mei},
       title={Averaging dynamics driven by fractional {B}rownian motion},
        date={2020},
        ISSN={0091-1798},
     journal={Ann. Probab.},
      volume={48},
      number={4},
       pages={1826\ndash 1860},
         url={https://doi.org/10.1214/19-AOP1408},
      review={\MR{4124526}},
}

\bib{harang2021regularity}{article}{
      author={Harang, Fabian~A},
      author={Ling, Chengcheng},
       title={Regularity of local times associated with {V}olterra--{L}{\'e}vy
  {P}rocesses and {P}ath-{W}ise {R}egularization of {S}tochastic {D}ifferential
  {E}quations},
        date={2021},
     journal={Journal of Theoretical Probability},
       pages={1\ndash 30},
}

\bib{doi:10.1142/S0219493721400104}{article}{
      author={Harang, Fabian~Andsem},
      author={Perkowski, Nicolas},
       title={C-infinity regularization of odes perturbed by noise},
        date={2021},
     journal={Stochastics and Dynamics},
      volume={0},
      number={0},
       pages={2140010},
      eprint={https://doi.org/10.1142/S0219493721400104},
         url={https://doi.org/10.1142/S0219493721400104},
}

\bib{MR3617205}{book}{
      author={Hyt\"{o}nen, Tuomas},
      author={van Neerven, Jan},
      author={Veraar, Mark},
      author={Weis, Lutz},
       title={Analysis in {B}anach spaces. {V}ol. {I}. {M}artingales and
  {L}ittlewood-{P}aley theory},
      series={Ergebnisse der Mathematik und ihrer Grenzgebiete. 3. Folge. A
  Series of Modern Surveys in Mathematics [Results in Mathematics and Related
  Areas. 3rd Series. A Series of Modern Surveys in Mathematics]},
   publisher={Springer, Cham},
        date={2016},
      volume={63},
        ISBN={978-3-319-48519-5; 978-3-319-48520-1},
      review={\MR{3617205}},
}

\bib{MR3752640}{book}{
      author={Hyt\"{o}nen, Tuomas},
      author={van Neerven, Jan},
      author={Veraar, Mark},
      author={Weis, Lutz},
       title={Analysis in {B}anach spaces. {V}ol. {II}},
      series={Ergebnisse der Mathematik und ihrer Grenzgebiete. 3. Folge. A
  Series of Modern Surveys in Mathematics [Results in Mathematics and Related
  Areas. 3rd Series. A Series of Modern Surveys in Mathematics]},
   publisher={Springer, Cham},
        date={2017},
      volume={67},
        ISBN={978-3-319-69807-6; 978-3-319-69808-3},
         url={https://doi.org/10.1007/978-3-319-69808-3},
        note={Probabilistic methods and operator theory},
      review={\MR{3752640}},
}

\bib{MR4089788}{article}{
      author={L\^{e}, Khoa},
       title={A stochastic sewing lemma and applications},
        date={2020},
     journal={Electron. J. Probab.},
      volume={25},
       pages={Paper No. 38, 55},
         url={https://doi.org/10.1214/20-ejp442},
      review={\MR{4089788}},
}

\bib{le2021taming}{article}{
      author={L{\^e}, Khoa},
      author={Ling, Chengcheng},
       title={Taming singular stochastic differential equations: {A} numerical
  method},
        date={2021},
     journal={arXiv preprint arXiv:2110.01343},
}

\bib{MR4261328}{article}{
      author={Liu, Chong},
      author={Pr\"{o}mel, David~J.},
      author={Teichmann, Josef},
       title={Stochastic analysis with modelled distributions},
        date={2021},
        ISSN={2194-0401},
     journal={Stoch. Partial Differ. Equ. Anal. Comput.},
      volume={9},
      number={2},
       pages={343\ndash 379},
         url={https://doi.org/10.1007/s40072-020-00166-7},
      review={\MR{4261328}},
}

\bib{li2020slowfast}{article}{
      author={Li, Xue-Mei},
      author={Sieber, Julian},
       title={Slow-fast systems with fractional environment and dynamics},
        date={2020},
     journal={arXiv preprint arXiv:2012.01910},
}

\bib{MR1654527}{article}{
      author={Lyons, Terry~J.},
       title={Differential equations driven by rough signals},
        date={1998},
        ISSN={0213-2230},
     journal={Rev. Mat. Iberoamericana},
      volume={14},
      number={2},
       pages={215\ndash 310},
         url={https://doi.org/10.4171/RMI/240},
      review={\MR{1654527}},
}

\bib{MR242239}{article}{
      author={Mandelbrot, Benoit~B.},
      author={Van~Ness, John~W.},
       title={Fractional {B}rownian motions, fractional noises and
  applications},
        date={1968},
        ISSN={0036-1445},
     journal={SIAM Rev.},
      volume={10},
       pages={422\ndash 437},
         url={https://doi.org/10.1137/1010093},
      review={\MR{242239}},
}

\bib{MR3617459}{book}{
      author={Pisier, Gilles},
       title={Martingales in {B}anach spaces},
      series={Cambridge Studies in Advanced Mathematics},
   publisher={Cambridge University Press, Cambridge},
        date={2016},
      volume={155},
        ISBN={978-1-107-13724-0},
      review={\MR{3617459}},
}

\bib{MR875011}{incollection}{
      author={Rubio~de Francia, Jos\'{e}~L.},
       title={Martingale and integral transforms of {B}anach space valued
  functions},
        date={1986},
   booktitle={Probability and {B}anach spaces ({Z}aragoza, 1985)},
      series={Lecture Notes in Math.},
      volume={1221},
   publisher={Springer, Berlin},
       pages={195\ndash 222},
         url={https://doi.org/10.1007/BFb0099115},
      review={\MR{875011}},
}

\bib{MR1725357}{book}{
      author={Revuz, Daniel},
      author={Yor, Marc},
       title={Continuous martingales and {B}rownian motion},
     edition={Third},
      series={Grundlehren der Mathematischen Wissenschaften [Fundamental
  Principles of Mathematical Sciences]},
   publisher={Springer-Verlag, Berlin},
        date={1999},
      volume={293},
        ISBN={3-540-64325-7},
         url={https://doi.org/10.1007/978-3-662-06400-9},
      review={\MR{1725357}},
}

\bib{young}{article}{
      author={Young, L.~C.},
       title={An inequality of the {H}\"older type, connected with {S}tieltjes
  integration},
        date={1936},
        ISSN={0001-5962},
     journal={Acta Math.},
      volume={67},
      number={1},
       pages={251\ndash 282},
         url={http://dx.doi.org.www2.lib.ku.edu:2048/10.1007/BF02401743},
      review={\MR{1555421}},
}

\end{biblist}
\end{bibdiv}

\end{document}